\documentclass[journal]{IEEEtran}

\usepackage{macros}         

\begin{document}

\title{\vspace{-.2cm}
Achieving Linear Convergence in Distributed Asynchronous Multi-agent Optimization}

\author{Ye Tian, Ying Sun,   and Gesualdo Scutari\vspace{-.9cm}
\thanks{  Part of this work has been presented at the 56th Annual Allerton Conference 
\cite{Tian_Allerton18} and posted on arxiv \cite{Tian_arxiv} on March 2018.   This work  has been supported   by the USA National Science Foundation
under Grants CIF 1632599 and CIF 1719205; and in part by the Office of
Naval Research under Grant N00014-16-1-2244, and the Army Research Office under Grant W911NF1810238.}
\thanks{The authors are with   the School of Industrial Engineering, 
Purdue University, West-Lafayette, IN, USA; Emails:  \texttt{\{tian110,sun578,gscutari\}@purdue.edu.}}}

\maketitle

\begin{abstract}
 This papers studies  multi-agent (convex and \emph{nonconvex}) optimization over static digraphs. We propose a general    distributed \emph{asynchronous} algorithmic framework whereby   i) agents  can update   their local variables  as well as communicate with their neighbors at any time, without any form of coordination; and  ii) they can perform their local computations  using (possibly) delayed, out-of-sync information from  the other agents.  Delays need not be known to the agent or obey any specific profile, and can also be time-varying (but bounded).   
 The algorithm builds on a  tracking mechanism that is robust against asynchrony (in the above sense), whose goal is to estimate locally the average of agents' gradients. 
 When applied to strongly convex functions,   we prove that it converges at an R-linear (geometric)  rate as long as the step-size is  {sufficiently small}. A sublinear convergence rate is proved, when nonconvex problems and/or diminishing, {\it uncoordinated}  step-sizes are considered. 
  To the best of our knowledge, this is the first distributed algorithm with provable geometric convergence rate in such  a general asynchronous setting. 
  Preliminary numerical results demonstrate the efficacy of the proposed algorithm and validate our theoretical findings.\vspace{-0.2cm}
\end{abstract}

\begin{IEEEkeywords}
 Asynchrony, Delay, Directed graphs, Distributed optimization, Linear convergence, Nonconvex optimization.\vspace{-0.4cm}
\end{IEEEkeywords}

%
\IEEEpeerreviewmaketitle

\section{Introduction}
\label{sec:intro}

We study convex and nonconvex distributed optimization over a network of agents, modeled as a directed fixed graph. Agents aim at cooperatively solving the   optimization problem\vspace{-0.1cm}
\begin{equation}
\tag{P}
  \min_{\bx\in \mathbb{R}^n}   F ( \bx) \triangleq \sum_{i=1}^I f_i\big(\bx\big)  
\label{eq:problem} \vspace{-0.1cm}
\end{equation}
where   $\map{f_i}{\real^{n}}{\real}$ is the cost
function of agent $i$, assumed to be smooth (nonconvex) and known only to agent $i$. In this setting, 
optimization has to be performed in a distributed, collaborative manner:  agents can only receive/send  information from/to its immediate neighbors.  Instances of  \eqref{eq:problem}  that require distributed computing  have found a wide range of applications in different areas, including 
network information processing, 
resource allocation in communication networks, 
swarm robotic, and machine learning, just to name a few. 
 
 Many of the aforementioned applications give rise to extremely large-scale
problems and networks, which naturally call for {\it asynchronous}, {\it parallel} solution methods. In
fact, 
asynchronous modus operandi 
reduces the idle times of workers, mitigate communication and/or memory-access
congestion, save power (as agents need not perform computations and communications at every iteration), and make algorithms more fault-tolerant. 
 In this paper, we consider the following  very general, abstract,   asynchronous model \cite{Bertsekas_Book-Parallel-Comp}:   \begin{description}
    \item[(i)] 	Agents  can perform  their local computations  as well as communicate (possibly in parallel) with their immediate neighbors at any time, without any form of coordination or centralized scheduling; and %
   \item[(ii)] when solving their local subproblems, agents can use outdated information from  their neighbors.  \end{description}
   In (ii) no   constraint is imposed on the delay profiles: delays  can be arbitrary (but bounded),    time-varying, and (possibly)  dependent on the specific activation rules adopted to wakeup the agents in the network.
This model captures in a unified fashion several  {forms} of asynchrony: some  {agents} execute more iterations than others; some agents communicate more frequently than others; and inter-agent communications can be unreliable and/or subject to  unpredictable, time-varying delays. 
 
 Several forms of asynchrony have been studied in the literature--see Sec.~\ref{sec:state_of_the_art} for an overview of   related works. However, 
 we are not aware of any   distributed algorithm that is compliant to  the  asynchrony model (i)-(ii) and distributed (nonconvex) setting above. Furthermore, when considering the special case of a strongly convex function $F$, 
   it is  not clear how to   design    a  (first-order)   {\it distributed asynchronous}  algorithm (as specified above) that achieves {\it linear convergence}  rate.  
 { This paper answers these questions}--
 see Sec.~\ref{contributions} and {Table~1} for a summary  of our  contributions.\vspace{-0.3cm}

\begin{table*}[t]\label{tb:compare_final}\vspace{-0.3cm}
\centering
\resizebox{\textwidth}{!}{
\begin{tabular}{|c | c | c | c | c | c | c | c | c | c | c| }
\hline
\multirow{3}{*}{Algorithm }  &  \multirow{3}{6.5em}{\centering Nonconvex Cost Function}   &\multirow{3}{3.5em}{\centering No Idle Time }     &\multirow{3}{ 5em}{\centering Arbitrary Delays}    &\multirow{3}{*}{ Parallel}     & \multicolumn{2}{c|}{Step Sizes  } & \multirow{3}{*}{Digraph } &\multirow{3}{*}{Global  Convergence  to Exact Solutions} &\multicolumn{2}{c|}{Rate Analysis}  \\                        
\cline{6-7}   \cline{10-11}                                               
&                                      &                                                                                   &                                                                  &                                                                   &    \multirow{2}{*}{Fixed}    &           \multirow{2}{7em}{\centering Uncoordinated Diminishing }        &                        &                              &  \multirow{2}{8em}{Linear Rate for Strongly Convex}  & \multirow{2}{*}{Nonconvex}	\\
                                                                             &                                    &             &             &                                                        &                        &                   &                                        &               &  	& \\
\hline
\hline  
Asyn. Broadcast \cite{nedic2011asynchronous}			& 	                                  &                           &                                & \checkmark                         &\checkmark            & \checkmark                  &                           &   In expectation (w. diminishing step)                      &      & \\ 
\hline  
Asyn. Diffusion \cite{zhao2015asynchronous}			& 	                                  &                           &                                &                          &\checkmark            &                   &                           &                         &    & \\ 
\hline
Asyn. ADMM \cite{kumar2017asynchronous}			& 	 \checkmark                                 &                           &                                 &                          &  \checkmark           &                &                          & Deterministic            &  &   \\ 
\hline
Dual Ascent in \cite{eisen2017decentralized}			&                        	       & \checkmark          &     Restricted                 &     Restricted         & \checkmark              &              &                              &                              &  &  \\ 
\hline
ra-NRC            \cite{bof2017newton}                           &                                      &                            &                                &                           &  \checkmark             &               & \checkmark             &                          &  & \\   
\hline
ARock             \cite{peng2016arock}                     		&                               & \checkmark                 &  Restricted                 &                   &  \checkmark            &                &                         & Almost surely              &  In expectation & \\ 
\hline
ASY-PrimalDual \cite{wu2016decentralized}                    &                                    & \checkmark                 &  Restricted             &                   &  \checkmark           &               &                        & Almost surely                 &  & \\ 
\hline
\textbf{ASY-SONATA}                                              &  \checkmark                  & \checkmark                    & \checkmark             &   \checkmark      &\checkmark           &  \checkmark    &\checkmark               &Deterministic      & Deterministic   &   Deterministic \\ 
\hline
\end{tabular}
}\smallskip

 \caption{ C\MakeLowercase{omparison with state-of-art distributed asynchronous algorithms}. Current schemes can deal with uncoordinated activations but only with some    forms of delays.  ASY-SONATA enjoys all the desirable features listed in the table.} \vspace{-0.3cm}
\vspace{-0.5cm}
\end{table*}

    \subsection{Literature Review}\label{sec:state_of_the_art}
 Since the seminal work 
 \cite{tsitsiklis1986distributed},  asynchronous parallelism has been applied to several    {\it centralized} optimization
algorithms, including block coordinate descent  (e.g., \cite{tsitsiklis1986distributed,liu2015asyspcd,cannelli2016asynchronous}) and  stochastic gradient   (e.g., \cite{Hogwild!,lian2015asynchronous}) methods. However, these schemes are not  applicable to    the networked setup considered in this paper, because they  would require the knowledge of  the    function $F$ from each agent. 
{\it Distributed} methods exploring (some form of) asynchrony   over networks with no centralized node  have been studied in    \cite{NotarnicolaNotarstefanoTAC17,xu2017convergence,iutzeler2013asynchronous,wei20131,bianchi2016coordinate,wang2015cooperative,li2016distributed,Tsianos:ef,Tsianos:en,lin2016distributed,doan2017impact,nedic2011asynchronous,zhao2015asynchronous,kumar2017asynchronous,eisen2017decentralized,bof2017newton,peng2016arock,wu2016decentralized}.   
We group next these works  based upon the  features (i)-(ii) above.  
  
\noindent\textbf{(a)~Random activations and no delays} \cite{NotarnicolaNotarstefanoTAC17,xu2017convergence,iutzeler2013asynchronous,wei20131,bianchi2016coordinate}: These schemes considered distributed {\it convex} unconstrained optimization over {\it undirected} graphs. While substantially different in the form of the updates performed by the agents--\cite{NotarnicolaNotarstefanoTAC17,iutzeler2013asynchronous,bianchi2016coordinate} are   instances of  primal-dual (proximal-based) algorithms, \cite{wei20131} is an ADMM-type algorithm, while  \cite{xu2017convergence}  is based on the  distributed gradient tracking mechanism introduced in\cite{DiLorenzoCAMSAP15,xu2015augmented,di2016next}--all these algorithms   are asynchronous in the sense of feature (i) [but not (ii)]: at each iteration, a subset of agents \cite{NotarnicolaNotarstefanoTAC17,iutzeler2013asynchronous,bianchi2016coordinate} (or edge-connected agents \cite{wei20131,xu2017convergence}), chosen at random, is activated,  performing then  their updates and communications with their immediate neighbors; between two activations, agents are  assumed to be  in {\it idle} mode (i.e., able to {\it continuously} receive information).  However,  {\it no form of  delays} is allowed:  every agent must perform its local computations/updates using the {\it most updated} information from its neighbors.   
This means that  all the actions performed by the agent(s) in   an activation must be completed before a new activation (agent)   takes place (wakes-up), which calls for some coordination among the agents. 
 {Finally,   no convergence rate was  provided for the aforementioned schemes but  \cite{wei20131,xu2017convergence}.}  \\   
  \noindent\textbf{(b)~Synchronous activations  and delays} \cite{wang2015cooperative,li2016distributed,Tsianos:ef,Tsianos:en,lin2016distributed,doan2017impact}: 
  These schemes considered  distributed constrained {\it convex} optimization over {\it undirected} graphs. They study the impact of delayed gradient information 
  \cite{wang2015cooperative,li2016distributed} or   communication delays (fixed \cite{Tsianos:ef}, uniform \cite{doan2017impact,li2016distributed} or time-varying \cite{Tsianos:en,lin2016distributed}) on the convergence rate of distributed gradient (proximal \cite{wang2015cooperative,li2016distributed} or projection-based \cite{lin2016distributed,doan2017impact}) algorithms or  dual-averaging distributed-based  schemes \cite{Tsianos:ef,Tsianos:en}. While these schemes are all synchronous [thus lacking of feature (i)], they can tolerate {\it communication delays} [an instantiation of feature (ii)],  converging at a {\it sublinear rate} to an optimal solution.   
Delays   must be  such that no  losses   occur--every agent's message  {will eventually} reach its destination within a finite {time}. \\
  \noindent\textbf{(c) Random/cyclic activations and some form of delays} \cite{nedic2011asynchronous,zhao2015asynchronous,kumar2017asynchronous,eisen2017decentralized,bof2017newton,peng2016arock,wu2016decentralized}:  The class of optimization problems along with the key features of the algorithms proposed  in these papers  are summarized in {Table~1} and briefly discussed next. The majority of these  works studied distributed (strongly) {\it convex} optimization over {\it undirected} graphs, with   \cite{zhao2015asynchronous} assuming that all the functions $f_i$ have the same minimizer, \cite{kumar2017asynchronous} considering also nonconvex objectives,  and \cite{bof2017newton} being implementable also over digraphs. The algorithms in 
  \cite{nedic2011asynchronous,zhao2015asynchronous} are  gradient-based schemes; \cite{kumar2017asynchronous} is a decentralized instance of  ADMM; \cite{peng2016arock}   applies an asynchronous parallel ADMM scheme  to distributed  optimization; and \cite{wu2016decentralized} builds on a primal-dual method. The schemes in  \cite{eisen2017decentralized,bof2017newton} instead build on (approximate) second-order information. All these algorithms are asynchronous in the sense of feature (i): \cite{nedic2011asynchronous,zhao2015asynchronous,kumar2017asynchronous, peng2016arock,wu2016decentralized} considered random activations of the agents (or edges-connected agents) while \cite{eisen2017decentralized,bof2017newton} studied deterministic, uncoordinated activation rules. As far as feature (ii) is concerned, some    form of delays is allowed. 
 %
 %
   More specifically, \cite{nedic2011asynchronous,zhao2015asynchronous,kumar2017asynchronous,bof2017newton} can deal    with {\it packet losses}:   the information sent by an agent to its neighbors either gets lost or received  with {\it no delay}. They also assume  that agents are {\it always in idle mode} between two activations. Closer to the proposed asynchronous framework  are the schemes in  \cite{peng2016arock,wu2016decentralized} wherein a probabilistic model   is employed  to describe the activation of the agents and the aged information used in their updates.   The
 model  requires that the random variables triggering  the
activation of the agents are i.i.d and {\it independent} of  the delay vector used by the agent to performs its update. 
While this assumption makes the  convergence analysis possible, in reality, there is a strong dependence of the delays   
on the activation index; see \cite{cannelli2016asynchronous} for a detailed discussion on this issue and several
   counter examples. Other consequences of this model are: the schemes \cite{peng2016arock,wu2016decentralized} are {\it not parallel}--only one agent per time can perform the update--and a random self-delay must be used  in the update of each agent (even if agents have access to their most recent information).  Furthermore, \cite{peng2016arock} calls for the solution of a convex subproblem for each agent at every iteration. Referring to the convergence rate,  \cite{peng2016arock} is the only scheme exhibiting linear convergence  {\it in expectation}, 
   when each $f_i$ is strongly convex and the graph  {\it undirected}. 
   No convergence rate is available in any of the aforementioned papers, when $F$  is nonconvex. \vspace{-.4cm}

 \subsection{Summary of Contributions}\label{contributions}\vspace{-.1cm} 
 This paper proposes a general distributed, asynchronous   algorithmic framework   for (strongly) convex and {\it nonconvex} instances of Problem~\eqref{eq:problem}, over {\it directed} graphs. The algorithm leverages  a  perturbed ``sum-push'' mechanism that is robust against asynchrony, whose goal is to track locally the average of agents' gradients; this scheme along with its convergence analysis are of independent interest. 
 To the best of our knowledge, the proposed framework is  the first  scheme combining the following attractive features (cf. {Table~1}): 
(a) it is {\it parallel and asynchronous [in the sense (i) and (ii)]}--multiple agents can be activated at the same time (with no coordination) and/or outdated information can be used in the agents' updates; our  asynchronous setting (i) and (ii) is  less restrictive than the one in \cite{peng2016arock,wu2016decentralized}; furthermore, in contrast with \cite{peng2016arock}, our scheme avoids solving possibly complicated subproblems; (b) it is applicable to  {\it nonconvex} problems, with probable convergence to  stationary solutions of \eqref{eq:problem};  (c) it is implementable over {\it digraph};  (d) it employs either a constant step-size  or {\it uncoordinated} diminishing ones; (e) it  \textit{converges at an R-linear rate} (resp. sublinear)   when $F$ is strongly convex  (resp. nonconvex) and a constant (resp. diminishing, uncoordinated) step-size(s)
 is employed; this contrasts \cite{peng2016arock} wherein each $f_i$ needs to be strongly convex; and (f) it is ``protocol-free'', meaning that  agents need not obey any  specific   communication protocols or  asynchronous modus operandi (as long as delays are bounded and agents  update/communicate uniformly infinitely often). 
 
 On the technical side, convergence is studied introducing two techniques of independent interest, namely: i)  the asynchronous agent system   is reduced to   a synchronous ``augmented'' one with no delays by  adding  virtual agents to the graph.  While this idea was first explored in  \cite{dominguez2011distributed1,nedic2010convergence},\cite{lin2013constrained}, the proposed enlarged system and algorithm   differ from   those used therein, which cannot deal with the general asynchronous model considered here--see   Remark \ref{rmk_augmented_graph}, Sec.\ref{sec:pertavg}; and ii) the rate analysis is employed putting forth a generalization of the small gain theorem (widely used in the literature \cite{Nedich-geometric}   to analyze synchronous schemes), which  is expected to be broadly applicable to other distributed algorithms.\vspace{-0.2cm}
 
 \subsection{Notation} 
Throughout the paper we use  the following notation.  Given the matrix   $\mathbf{M}\triangleq (M_{ij})_{i,j=1}^I$, 
$\mathbf{M}_{i,:}$ and $\mathbf{M}_{:,j}$ denote its $i$-th row vector and $j$-th column vector. Given the  sequence $\{ \bd{M}^t \}_{t=s}^k$, with $k\geq s$, we define   $\bd{M}^{k:s} \triangleq \bd{M}^k \bd{M}^{k-1} \cdots \bd{M}^{s+1} \bd{M}^s$, if $k>s$; and $\bd{M}^{k:s} \triangleq   \bd{M}^s$ otherwise. Given two matrices (vectors) $\mathbf{A}$ and $\mathbf{B}$ of same size,  by $\mathbf{A} \preccurlyeq \mathbf{B}$ we mean  that $\mathbf{B}-\mathbf{A}$   is a nonnegative matrix (vector). The dimensions of the all-one vector $\bd{1}$ and the $i$-th canonical vector $\mathbf{e}_i$ will be clear from the context. We use $\norm{\cdot}$ to represent the Euclidean norm for a vector whereas the spectral norm for a matrix.   
The indicator function $\mathbbm{1}[E]$ of an event $E$ equals to $1$ when the event $E$ is true,  and $0$ otherwise. 
Finally, we use the convention  $\sum_{t \in \emptyset} x^t = 0$ and $\prod_{t \in \emptyset} x^t = 1$. 
 
 
 
 
 \section{Problem Setup and Preliminaries}\label{sec:Setup}
 \subsection{Problem Setup}\label{sec:Problem_setup}
 We study Problem~\eqref{eq:problem} under the following   assumptions.   \vspace{-0.1cm}\begin{assumption}[On the optimization problem]\mbox{}
\begin{enumerate} 
  \item Each $f_i:\mathbb{R}^{n}\rightarrow \mathbb{R}$  is proper, closed and  $L_i$-Lipschitz differentiable;
  \item $F$ is bounded from below.~\oprocend
\end{enumerate}
\label{ass:cost_functions}\vspace{-0.1cm}
\end{assumption}
Note that $f_i$ need not be convex. We also make the blanket   assumption  that each agent $i$ knows only its own  $f_i$,  but not $\sum_{j\neq i} f_j$.  To state linear convergence, we will use  the following extra condition on the objective function. \vspace{-0.1cm}
  \begin{assumption}[Strong convexity] Assumption~\ref{ass:cost_functions}{(i)} holds  and, in addition,  
  $F$  is $\tau$-strongly convex.~\oprocend 
\label{ass:cost_functions_strongly_convex}\vspace{-0.1cm}
\end{assumption} 
 \noindent \textbf{On the communication network:} The  communication network of the agents is  modeled as a  fixed, directed
graph $\GG = (\mathcal{V},\EE)$, where $\mathcal{V}=\until{I}$ is the set of nodes (agents), and $\EE\subseteq \mathcal{V} \times \mathcal{V}$
is the set of edges ({communication} links). If  $(i,j)\in \EE$, it means that agent $i$
can send information to agent $j$. We assume that the digraph does not have self-loops.
We denote by $\nbrs_i^{\text{in}}$ the set of \emph{in-neighbors} of node $i$, i.e.,  
$\nbrs_i^{\text{in}} \triangleq \left\{j \in \mathcal{V} \mid (j,i) \in \EE \right\}$ while   $\nbrs_i^{\text{out}} \triangleq \left\{j \in \mathcal{V} \mid (i,j) \in \EE \right\}$
is the   set of \emph{out-neighbors} of agent $i$. 
We make the following standard assumption on the graph connectivity. 
\begin{assumption}
The graph $\GG$ is strongly connected.~\oprocend
\label{ass:strong_conn}
\end{assumption} 
 \vspace{-0.4cm}

 \subsection{Preliminaries: The SONATA algorithm \cite{sun2016distributed,YingMAPR}}\label{sec:SONATA}
The proposed asynchronous algorithmic framework   builds on the   synchronous SONATA algorithm,    proposed in~\cite{sun2016distributed,YingMAPR}  to solve (nonconvex) multi-agent optimization problems over time-varying digraphs. This is motivated by the fact that SONATA has the unique property of being provably applicable to both convex and nonconvex problems, and it achieves linear convergence when applied to strongly convex objectives $F$. We thus begin reviewing  SONATA, tailored to~\eqref{eq:problem};   then we generalized it  to the asynchronous setting (cf. Sec.~\ref{sec:Asy-SONATA}).

Every agent controls and iteratively updates the tuple  $(\bx_i, \by_i, {\mathbf{z}_i},\phi_i)$: $\bx_i$ is agent $i$'s copy of the shared variables $\bx$ in \eqref{eq:problem};  $\by_i$ acts as a local proxy of the sum-gradient  {$\nabla F$}; and     $\mathbf{z}_i$ and $\phi_i$  are auxiliary variables instrumental to deal with communications over  digraphs. Let $\bx_i^k, \mathbf{z}_i^k$, $\phi_i^k$, and $\mathbf{y}_i^k$ denote  the value of the aforementioned variables at iteration $k\in \mathbb{N}_0$. The  update of each agent $i$  reads: \vspace{-0.1cm}
\begin{align}
	\bx_i^{k+1} &= \sum_{j \in \mathcal{N}_i^{\text{in}}\cup \{i\}} w_{ij} \,{{ {\left(\bx_{j}^{k}-\alpha^k \,\by_{j}^{k}\right)}}   },  \label{eq:SONATA_x_update}\\
	\mathbf{z}_i^{k+1} & = \sum_{j \in \mathcal{N}_i^{\text{in}}\cup \{i\}} a_{ij}  \mathbf{z}_j^k + \nabla f_i(\bx_i^{k+1}) -  \nabla f_i(\bx_i^{k}) , \label{eq:SONATA_s_update} \\
	\phi_i^{k+1} & = \sum_{j \in \mathcal{N}_i^{\text{in}}\cup \{i\}} a_{ij} \phi_j^k,  \label{eq:SONATA_phi_update}\\
	\mathbf{y}_i^{k+1} & = \mathbf{z}_i^{k+1}/\phi_i^{k+1}, \label{eq:SONATA_y_update}
\end{align}
with 
{$\mathbf{z}_i^0 = \by_i^0=\nabla f_i (\bx_i^0)$} and $\phi_i^0 = 1$, for all $i\in \mathcal{V}$. In \eqref{eq:SONATA_x_update}, {$\by_i^k$ is a local estimate of the average-gradient $({1}/{I}) \sum_{i=1}^I \nabla f_i(\bx_i^k)$.} Therefore, every agent, first moves along the estimated gradient direction, generating $\bx_i^{k}-\alpha^k \,\by_i^k$ ($\alpha^k$ is the step-size); and then performs a consensus step to   {force} asymptotic agreement among the local variables $\bx_i$. Steps (\ref{eq:SONATA_s_update})-(\ref{eq:SONATA_y_update}) represent a perturbed-push-sum update, aiming at tracking the   gradient {$({1}/{I})\,\nabla F$} \cite{di2016next, YingMAPR, xu2015augmented}.  The weight-matrices $\mathbf{W}\triangleq (w_{ij})_{i,j=1}^I$ and $\mathbf{A}\triangleq (a_{ij})_{i,j=1}^I$ satisfy the following standard assumptions. \vspace{-0.1cm}  
\begin{assumption}[On the weight-matrices]\label{assumption_weights} The weight-matrices $\mathbf{W}\triangleq (w_{ij})_{i,j=1}^I$ and $\mathbf{A}\triangleq (a_{ij})_{i,j=1}^I$ satisfy (we will write $\mathbf{M}\triangleq (m_{ij})_{i,j=1}^I$ to denote either $\mathbf{A}$ or $\mathbf{W}$):
	\begin{enumerate} 
  \item 
 {    {$\exists\, \bar{m}>0$} such that $m_{ii} \geq \bar{m}$,  for all $i \in \mathcal{V}$; and $m_{ij} \geq \bar{m}$,  for all $(j,i) \in \mathcal{E}$; $m_{ij}=0$, otherwise; }
  \item $\mathbf{W}$ is row-stochastic, that is,  $\mathbf{W}\,\mathbf{1}=\mathbf{1}$;
  \item $\mathbf{A}$ is column-stochastic, that is,  $\mathbf{A}^T\,\mathbf{1}=\mathbf{1}$;
~\oprocend
\end{enumerate}
\end{assumption}\smallskip

 In \cite{Nedich-geometric},  a special instance  of SONATA, was proved to converges  at an R-linear  rate when $F$ is strongly convex.  This result was   extended to constrained, nonsmooth (composite), distributed optimization in  \cite{DanSunScutari-complexity}. 
A natural question is   whether   SONATA   works also in an asynchronous setting still converging at 
a linear  rate. 
Naive asynchronization   of    the  updates \eqref{eq:SONATA_x_update}-(\ref{eq:SONATA_y_update})--such as using uncoordinated  activations and/or replacing instantaneous information with a  delayed one--would  not work. 
For instance, the   tracking    (\ref{eq:SONATA_s_update})-(\ref{eq:SONATA_y_update}) calls for the  invariance of the  averages, i.e., $\sum_{i=1}^I \mathbf{z}^{k}_i= \sum_{i=1}^I \nabla f_i(\mathbf{x}^{k})$, for all $k\in \mathbb{N}_0$.   It is not difficult to  check that any perturbation   in (\ref{eq:SONATA_s_update})-e.g., in the form of delays or packet losses--puts in  jeopardy this property. 

     To cope with the above challenges, a first step is robustifying the gradient  tracking scheme. In Sec.~\ref{sec:Asy-Sum-Push},   we introduce  {\PertAvgname/}--an asynchronous, perturbed, instance of the  push-sum algorithm \cite{Kempe_PushSum03}, which serves as a unified algorithmic framework to accomplish several  tasks over digraphs in an asynchronous manner, such as solving the average consensus problem and  tracking the average of agents' time-varying signals. 
     Building on {\PertAvgname/}, in Sec.~\ref{sec:Asy-SONATA}, we  finally present the proposed distributed asynchronous optimization framework, termed  ASY-SONATA. 
     \vspace{-0.2cm}
      
  \section{Perturbed  Asynchronous Sum-Push}\label{sec:Asy-Sum-Push} 
We present  {\PertAvgname/}; 
the algorithm was first introduced in our conference paper \cite{Tian_Allerton18}, which we refer to for   details on the genesis of the scheme and intuitions; 
here we directly introduce the scheme and study its convergence. \\\indent 
 Consider an asynchronous setting wherein     
 agents compute and communicate independently  without coordination. 
 Every agent $i$ maintains state variables   $\bz_{i}$, $\phi_i$, $\by_i$, along with the following auxiliary  variables that are instrumental to deal with uncoordinated activations and delayed information: i) the cumulative-mass variables $\bm \rho_{ j i}$ and  $\sigma_{ j i}$, with $j\in \mathcal{N}_i^{\text{out}}$, which capture the cumulative (sum) information generated by agent $i$ up to the current time and to be sent to   agent   $j\in \mathcal{N}_i^{\text{out}}$; consequently,   $\bm \rho_{ ij}$ and  $\sigma_{ij}$ are received by $i$ from  its in-neighbors   $j\in \mathcal{N}_i^{\text{in}}$; and ii) the buffer variables  $\tilde{\bm \rho}_{ i j}$ and $\tilde{\sigma}_{ i j}$, with $j\in \mathcal{N}_i^{\text{in}}$,  which store the information sent from $j\in \mathcal{N}_i^{\text{in}}$ to $i$ and used by $i$ in its last update. Values of these variables at iteration $k\in \mathbb{N}_0$ are denoted by the same symbols with  the superscript ``$k$''. Note that, because of the asynchrony, each agent $i$ might have  outdated  $\bm \rho_{ij}$ and $\sigma_{ij}$;  $\bm \rho_{ij}^{k-d_j^k}$ (resp.  $\sigma_{ij}^{k-d_j^k}$)  is a delayed version of the current  $\bm\rho_{ij}^{k}$ (resp. $\sigma_{ij}^{k}$) owned by $j$ at time $k$, 
 where  
  $0\leq d_{j}^k\leq D<\infty$ is the delay. Similarly,   $\tilde{\bm \rho}_{ij}$ and  $\tilde{\sigma}_{ij}$ might differ from the last information generated by $j$ for $i$, because agent $i$ might not have received that information yet (due to delays)  or never will (due to packet losses).   \\ \indent The proposed asynchronous algorithm,     \PertAvgname/,    is summarized in Algorithm~\ref{alg:pertavg}. 
  A global iteration clock (not known to the agents) is introduced:  $k\to k+1$ is triggered 
 based upon the completion from  one agent, say  $i^k$,  of the following   actions. 
\textbf{{(S.2)}:}   agent $i^k$ maintains a local variable  {$\tau_{i^k j}$}, for each $j\in \mathcal{N}_{i^k}^{\text{in}}$, which  keeps track of the ``age'' (generated time) of the $({\bm  \rho}, \sigma)$-variables that it has received from its  in-neighbors and {\it already} used. 
   If  $k - d_j^k$ is larger than the current counter   
   $\tau_{i^k j}^{k-1}$, indicating that the received $({\bm \rho},\sigma)$-variables are newer than those currently stored, agent  $i^k$   accepts ${\bm \rho}_{i^k j}^{k - d_j^k}$ and ${\sigma}_{i^k j}^{k - d_j^k}$, and updates $\tau_{i^k j}$ as $k - d_j^k$; otherwise, the variables will be discarded and     $\tau_{i^k j}$ remains unchanged.   Note that  \eqref{eq:purge_old} can be performed without any coordination. It is sufficient that each agent attaches a time-stamp to its produced information     reflecting it local timing counter.   We describe next the other steps, assuming that new information has come in to agent $i^k$, that is, $\tau_{i^k j}=k - d_j^k$.
\begin{algorithm}[t]
\caption{\PertAvgname/ (Global View)}\label{alg:pertavg}
  \begin{algorithmic}
    \StatexIndent[0] \textbf{Data:}   $\mathbf{z}_i^0 \in \mathbb{R}^n$, $\phi_i^0 = 1$,  {$\tilde{\boldsymbol{\rho}}_{ij}^0 = 0$,}  $ \tilde{\sigma}_{ij}^0 = 0$, $\tau_{i j}^{-1} = -D$, for all  $j \in \mathcal{N}_i^{\text{in}}$ and $i\in \mathcal{V}$;   $\sigma_{ij}^t = 0$ and $\boldsymbol{\rho}_{ij}^t = 0$,   for all  $t = -D,  \ldots, 0$; and  $\{{\bm \epsilon}^k \}_{k\in \mathbb{N}_0}$.  Set $k = 0$.  
\While{a termination criterion is not met}
    \State  \texttt{(S.1)}  Pick $(i^k, \bd{d}^k)$, with $\mathbf{d}^k \triangleq (d_{j}^k)_{j \in \mathcal{N}_{{i}^k}^{\text{in}}}$;
    \State \texttt{(S.2)} Set (purge out the old information):
	\begin{align}\label{eq:purge_old}
        \tau_{i^k j}^k = \max\big(\tau_{i^k j}^{k-1}, k- d_j^k\big), \quad \forall j \in \mathcal{N}_{i^k}^{\text{in}};
        \end{align}
   \State \texttt{(S.3)}  Update the variables performing
   \begin{itemize}[leftmargin=1cm] \small 
     \item  \texttt{(S.3.1)} \textbf{Sum step:}
	\begin{align} 
		& \label{eq:pertavg_sum} \bz_{i^k}^{k+\frac{1}{2}}  = \bz_{i^k}^{k}  +  \displaystyle\sum_{j \in \mathcal{N}_{i^k}^{\text{in}}}  \left({\bm \rho}_{{i^k}j}^{\tau_{i^k j}^k }-\tilde{\bm \rho}_{i^kj}^k\right) + \bm{\epsilon}^k \\
		& \phi_{i^k}^{k+\frac{1}{2}}  = \phi_{i^k}^{k}  +  \displaystyle\sum_{j \in \mathcal{N}_{i^k}^{\text{in}}}  \left({\sigma}_{{i^k}j}^{\tau_{i^k j}^k }-\tilde{\sigma}_{i^kj}^k\right) \nonumber \end{align}\vspace{-0.3cm}
     \item \texttt{(S.3.2)} \textbf{Push step:}
		\begin{align}	
		& \bz_{i^k}^{k+1}    =  a_{i^ki^k}\, \bz_{i^k}^{k+\frac{1}{2}}  , \quad \phi_{i^k}^{k+1}    =  a_{i^ki^k}\, \phi_{i^k}^{k+\frac{1}{2}} \nonumber\\
		& {\bm \rho}_{j {i^k}}^{k+1}   = {\bm \rho}_{j {i^k}}^{k} + a_{j {i^k}} \,\bz_{i^k}^{k+\frac{1}{2}}, \label{eq:alg_rho}\\
		& {\sigma}_{j {i^k}}^{k+1}   = {\sigma}_{j {i^k}}^{k} + a_{j {i^k}} \,\phi_{i^k}^{k+\frac{1}{2}},\quad \forall j\in \mathcal{N}_{i^k}^{\text{out}}   \nonumber \end{align}		
     \item \texttt{(S.3.3)} \textbf{Mass-Buffer update:}
		\begin{align}	
		& \tilde{\bm \rho}_{{i^k}j}^{k+1} = {\bm \rho}_{{i^k}j}^{\tau_{i^k j}^k },  \quad   \tilde{\sigma}_{{i^k}j}^{k+1} = {\sigma}_{{i^k}j}^{\tau_{i^k j}^k } ,   { \quad \,\ \forall j \in \mathcal{N}_{i^k}^{\text{in}}} \label{eq:alg_buffer}
	\end{align}
	\item \texttt{(S.3.4)} \textbf{Set: }  $\quad \by_{i^k}^{k+1} =  \bz_{i^k}^{k+1} / \phi_{i^k}^{k+1} .$
	\end{itemize}
	 \State \texttt{(S.4)} \parbox[t]{\dimexpr\linewidth-\algorithmicindent}{ { Untouched state variables shift to state $k+1$ \\ while keeping the same value;    $k\leftarrow k+1$}.\strut}
\EndWhile
  \end{algorithmic} 
\end{algorithm} 
\textbf{{(S.3.1)}:}   In (\ref{eq:pertavg_sum}), agent $i^k$ builds the intermediate ``mass'' $\bz_{i^k}^{k+\frac{1}{2}}$ based upon its current information $\bz_{i^k}^{k}$ and $\tilde{\bm \rho}_{i^kj}^k$, and the (possibly) delayed one from its in-neighbors,  ${\bm \rho}_{{i^k}j}^{k-d_{j}^k}$; and     $\bm{\epsilon}^k \in \mathbb{R}^n$ is an exogenous perturbation (later this perturbation will be properly chosen to accomplish specific goals, see Sec.~\ref{sec:Asy-SONATA}).  Note that    the way  agent $i^k$   forms its
own estimates  $\bm \rho_{i^k j}^{k-d_j^k}$ is {\it immaterial} to the description of the
algorithm. 
The local buffer   $\tilde{\bm \rho}_{i^kj}^k$   stores the value of ${\bm \rho}_{i^kj}$ that agent $i^k$ used in its last update.  Therefore, if the information in  ${\bm \rho}_{i^kj}^{k-d_{j}^k}$ is not older than the one in  $\tilde{\bm \rho}_{i^kj}^{k}$, the    difference ${\bm \rho}_{i^kj}^{k-d_{j}^k} - \tilde{\bm \rho}_{i^kj}^{k}$ in (\ref{eq:pertavg_sum}) will capture  the  sum of the $a_{i^k j}\mathbf{z}_j$'s that  have been generated  by $j\in \mathcal{N}_{i^k}^{\text{in}}$ for $i^k$ up until $k-d_{j}^k$ 
and  not used   by agent $i^k$ yet. 
For instance, in a synchronous setting, one would have  {${\bm \rho}_{i^k j}^{{k}} - \tilde{\bm \rho}_{i^kj}^{k}=a_{i^kj}\mathbf{z}_j^{k+\frac{1}{2}}$}.  
\textbf{{(S.3.2)}:}    the generated   $\bz_{i^k}^{k+\frac{1}{2}}$   is   ``pushed back''  to    agent $i^k$ itself and its out-neighbors. Specifically,  out of  the total mass $\bz_{i^k}^{k+\frac{1}{2}}$ generated, agent $i^k$ gets $a_{ii}\,\bz_i^{k+\frac{1}{2}}$,   determining the update $\bz^k_i \to \bz_i^{k+1}$ while the remaining is allocated to the agents $j\in \mathcal{N}_{i^k}^{\text{out}}$, with  $\bm a_{ji^k}\,\bz_{i^k}^{k+\frac{1}{2}}$ cumulating to the  mass buffer  $\bm \rho_{ji^k}^{k}$ and  generating the update $\bm \rho_{ji^k}^{k}\to \bm \rho_{ji^k}^{k+1}$, to be sent to agent $j $. \textbf{{(S.3.3)}:}   each local buffer variable  $\tilde{\bm \rho}_{i^k j}^k$ is updated to account for the use of new information from $j\in\mathcal{N}_{i^k}^{\text{in}}$.  The final information is then read on the $\by$-variables [cf. \textbf{{(S.3.4)}}].   
\begin{remark}(Global view description) \emph{ Note that    each  agent's update     is fully defined, once $i^k$ and  {$\mathbf{d}^k$} are given.  
The selection $(i^k,\mathbf{d}^k)$ in \textbf{(S.1)} is not performed by anyone;   it is instead an {\it a-posteriori} description of    agents' actions: All agents act asynchronously and continuously; the agent completing the ``push'' step and updating its own variables triggers {\it retrospectively} the iteration counter  $k\to k+1$ and determines the pair  $(i^k,\mathbf{d}^k)$  along with  all quantities involved in the other steps. Differently from most of the current literature, this ``global view'' description of the   agents' actions   allows us  to 
 abstract from  specific 
 computation-communication protocols and asynchronous modus operandi and captures by a unified model a gamut of asynchronous schemes.  }   \end{remark}

  Convergence is given under the following assumptions. \vspace{-0.1cm}
\begin{assumption}[On the asynchronous model] \label{ass:delays}Suppose:
\begin{enumerate}
\item  $\exists$  $0<T<\infty$ such that  {$\cup_{t=k}^{k+T-1} i^t=\mathcal{V}$}, for all $k\in \mathbb{N}_0$; 	 
\item  $\exists$ $0<D<\infty$ such that $0 \leq d_j^k\leq D$, for all $j\in \mathcal{N}_{i^k}^{\text{in}}$ and $k\in \mathbb{N}_0$.\hfill $\square$
\end{enumerate} 
\end{assumption}
  
 The next  theorem   studies convergence  of  \PertAvgname/,  establishing geometric decay of the error $\|\mathbf{y}_i\spe{k} - (1 / I) \cdot \mathfrak{m}_z^{k}\|$, even in the presence of unknown (bounded) perturbations, where  $
\mathfrak{m}^{k}_{z} \triangleq  \sum_{i=1}^I \bd{z}_i^{k} + \sum_{(j,i)\in \mathcal{E}} (\bm{\rho}_{ij}^{k} - \tilde{\bm{\rho}}_{ij}^{k})$ represents   the ``total mass''  of the system at iteration $k$.\vspace{-0.1cm}

\begin{theorem}\label{thm:track}
 {Let $\{\mathbf{y}^k \triangleq [\mathbf{y}_1^k, \ldots,\mathbf{y}_I^k]^\top,$ $\mathbf{z}^k \triangleq [\mathbf{z}_1^k, \ldots,\mathbf{z}_I^k]^\top,$  $( \bm{\rho}_{ij}^{k}, \tilde{\bm{\rho}}_{ij}^{k})_{(j,i) \in \mathcal{E}} \}_{k \in \mathbb{N}_0}$  }be the sequence generated by Algorithm~\ref{alg:pertavg}, under Assumption~\ref{ass:strong_conn}, \ref{ass:delays}, and with   $\mathbf{A}\triangleq (a_{ij})_{i,j=1}^I$ satisfying Assumption~\ref{assumption_weights} (i),(iii). 
Define $K_1 \triangleq (2\,I-1)\cdot T+ I\cdot D .$  There exist   constants $\rho \in (0,1) $ and    $C_1>0$,  such that      \vspace{-0.1cm}
\begin{equation}\label{thm:track:eq}
\Big\|{\mathbf{y}_i\spe{k+1} - (1 / I) \cdot \mathfrak{m}_z^{k+1}}\Big\|\leq C_1 \left(\rho^{k}\norm{\bd{z}\spe{0}}+\sum_{l=0}^k \rho^{k-l}\norm{\bm{\epsilon}^l} \right),
\end{equation}  
for all $i\in\mathcal{V}$ and $k\geq K_1-1$. 

Furthermore,   $\mathfrak{m}^{k}_z = \sum_{i=1}^I {\bd{z}_i^0} + \sum_{t=0}^{k-1} \bm{\epsilon}^t$.
\end{theorem}
\begin{proof}
	See Sec.~\ref{sec:pertavg}.
\end{proof}

\noindent \textbf{Discussion:} Several comments are in order.

\subsubsection{On the asynchronous model}
Algorithm~\ref{alg:pertavg} captures a gamut of asynchronous {\it parallel} 
schemes and architectures, through the mechanism of generation of  $(i^k, \mathbf{d}^k)$.    Assumption~\ref{ass:delays} on $(i^k, \mathbf{d}^k)$ is quite mild: (a)  controls the frequency of the updates whereas (b) 
  limits the age of the old information used in the
computations; they can be easily enforced in practice.  For instance,   (a)  is readily satisfied if each agent wakes up and performs an update whenever some independent internal clock ticks or it is triggered by some of the neighbors;  (b) imposes conditions on the frequency and quality of the communications: information used by each agent cannot become infinitely old, implying that successful communications must occur sufficiently often. This however does not enforce any specific protocol on the activation/idle
time/communication.  For instance, i) agents need not perform the actions in  {Algorithm~\ref{alg:pertavg}} sequentially or inside the same activation round; or ii) executing the ``push'' step does not mean that agents must broadcast their new variables in the same activation; this would just incur a delay (or packet loss) in the  communication.\\\indent 
Note that the time-varying nature of the delays $\mathbf{d}^k$ permits to model also packet losses, as detailed next. Suppose that at iteration $k_1$ agent $j$ sends  its current $\rho,\sigma$-variables to its out-neighbor  $\ell$ and they  get lost; and let  $k_2$ be  the subsequent   iteration when $j$ updates again.   Let $t$ be the first iteration after $k_1$ when agent $\ell$ performs its update;  it  will use   information from $j$ such that  $t-d_j^t\notin [k_1+1,k_2]$, for some   $d_j^t\leq D<\infty$. If $t-d_j^t<  k_1+1$,     no newer information from $j$    has been  used by $\ell$; otherwise  $t-d_j^t\geq k_2+1$ (implying $k_2<t$), meaning that agent $\ell$ has used information  not older than  $k_2+1$. 

 \subsubsection{Comparison with \cite{dominguez2011distributed1,BofCarli17,bof2017newton}}   The use of 
  counter variables [such as  $({\bm \rho},\sigma,\tilde{{\bm \rho}},\tilde{\sigma})$-variables in our scheme]   was first introduced in  \cite{dominguez2011distributed1} to design a synchronous average consensus algorithm robust to packet losses. In \cite{BofCarli17}, this scheme was extended to deal with uncoordinated (deterministic) agents' activations  whereas  \cite{bof2017newton} built on \cite{BofCarli17} to design, in the same setting,  a  distributed Newton-Raphson  algorithm.  There are    important differences between  \PertAvgname/ and the aforementioned schemes, namely: i) none of them  can deal with \emph{delays but packet losses}; ii)  \cite{dominguez2011distributed1}  is {\it synchronous}; and  iii)\cite{BofCarli17,bof2017newton}   are not \emph{parallel} schemes, as  at each iteration only one agent is allowed to  wake up and transmit information to its neighbors. For instance,   \cite{BofCarli17,bof2017newton} cannot model synchronous parallel (Jacobi) updates.   Hence, the convergence analysis of \PertAvgname/  calls for a new line of proof, as introduced in Sec.~\ref{sec:pertavg}.  

 \subsubsection{Beyond average consensus} By choosing properly the perturbation signal ${\bm \epsilon}^k$,   \PertAvgname/ can solve different  problems. Some examples are discussed next.  \\
\noindent  {\it (i) Error free:  $\bm \epsilon^k=\mathbf{0}$.}   \PertAvgname/ solves the average consensus problem and \eqref{thm:track:eq} reads\vspace{-0.2cm}$$\Big\|{\mathbf{y}_i\spe{k+1} - \left(1/I\right) \cdot  \sum_{i=1}^I \bd{z}_i^0}\Big\|\leq C_1\,\rho^{k}\,\norm{\bd{z}\spe{0}}.\vspace{-0.2cm}$$

\noindent  {\it (ii) Vanishing error: $\lim_{k\to\infty} \|\bm{\epsilon}^k\| = 0$.}  Using \cite[Lemma~7(a)]{di2016next},   \eqref{thm:track:eq} reads    
 $\lim_{k\to\infty}\|\mathbf{y}_i\spe{k+1} -  \mathfrak{m}_z^{k+1}\|=0$. 
 
\noindent  {\it (iii) Asynchronous tracking.} 
Each agent $i$ owns a (time-varying) signal  $\{\bd{u}_i^k\}_{k\in \mathbb{N}_0}$; the average tracking problem consists in asymptotically track the average signal  {$\bar{\bd{u}}^k\triangleq (1/I)\cdot\sum_{i = 1}^I \mathbf{u}_i^k$,} that is, \vspace{-0.2cm}\begin{equation}\label{eq:signal}\lim_{k\to\infty}\norm{\by_i\spe{k+1} -  \bar{\mathbf{u}}^{k+1}} = 0,\quad \forall i\in \mathcal{V}.\end{equation} 
Under mild conditions on the signal, this can be accomplished in a distributed and asynchronous fashion, using \PertAvgname/, as formalized next.  \vspace{-0.2cm}
\begin{corollary}\label{cor:signal}
Consider, the following setting   in \PertAvgname/:
 $\bd{z}_i^0 = \bd{u}_i^0$, for all $i\in \mathcal{V}$;     $\bm{\epsilon}^k= \bd{u}_{i^k}^{k+1} - \tilde{\bd{u}}_{i^k}^{k}$, with \vspace{-0.3cm} 
\begin{align*}
\tilde{\bd{u}}_i^{k+1} = 
\begin{cases}
\bd{u}_i^{k+1} & \text{if } i=i^k;\\
\tilde{\bd{u}}_{i}^{k} & \text{otherwise};
\end{cases} \qquad \tl{\bd{u}}_i^{0} = \bd{u}_i^0;
\end{align*}
  Then     \eqref{thm:track:eq} holds, with  $\mathfrak{m}_z^{k+1}= \sum_{i = 1}^I \tilde{\bd{u}}_{i}^{k+1}$. Furthermore,    if   $\lim_{k\to\infty} \sum_{i=1}^I \norm{\bd{u}_i^{k+1}-\bd{u}_i^{k}} = 0$, then   \eqref{eq:signal} holds. 
\end{corollary}
\begin{proof}
	 See Appendix~\ref{pf:signal}.
\end{proof}
This instance of \PertAvgname/ will be used in  Sec.~\ref{sec:Asy-SONATA} to perform asynchronous gradient tracking. 
\begin{remark}[Asynchronous average consensus] \emph{To the best of our knowledge, the error-free instance of the \PertAvgname/ discussed above  is the first  (stepsize-free) scheme that   provably solves the  {\it average} consensus problem at a linear rate, under the general  asynchronous model described  by Assumption~\ref{ass:delays}. In fact, the existing    asynchronous consensus schemes  \cite{nedic2010convergence}  \cite{lin2013constrained} achieve an agreement among the agents' local variables whose value   is not in general the average of their initial values, but instead some  {\it unknown} function of them and the asynchronous modus operandi of the agents.
  Related to the \PertAvgname/ is the  ra-AC algorithm in \cite{BofCarli17}, which enjoys the same convergence property but under a more restrictive and specific asynchronous model (no delays but packet losses and single-agent activation per iteration).}\vspace{-0.2cm}  
\end{remark}
 
  \vspace{-0.3cm}

\section{Asynchronous SONATA (ASY-SONATA)}\label{sec:Asy-SONATA} 

 We  are ready now to introduce our distributed asynchronous algorithm--ASY-SONATA. The algorithm combines SONATA (cf.~Sec.~\ref{sec:SONATA}) with  \PertAvgname/ (cf.~Sec.~\ref{sec:Asy-Sum-Push}), the latter replacing the synchronous  tracking scheme \eqref{eq:SONATA_s_update}-\eqref{eq:SONATA_y_update}. The ``global view'' of the scheme is given in Algorithm~\ref{alg:AsyTracking}. 

 \begin{algorithm}[!ht]
\caption{ASY-SONATA (Global View)}
  \begin{algorithmic}
    \StatexIndent[0] \textbf{Data:} For all agent $i$ and $\forall j \in \mathcal{N}_i^{\text{in}}$, $\mathbf{x}_i^0 \in \mathbb{R}^n$, $\mathbf{z}_i^0 = \nabla f_i(\mathbf{x}_i^0)$, $\phi_i^0 = 1$,  {$\tilde{\boldsymbol{\rho}}_{ij}^0 = 0$},  $ \tilde{\sigma}_{ij}^0 = 0$, $\tau_{i j}^{-1} = -D$.  And for $t = -D, -D+1, \ldots, 0$, $\boldsymbol{\rho}_{ij}^t = 0$, $\sigma_{ij}^t = 0$, $\mathbf{v}_i^t = 0$.  Set $k = 0$.

    \While{a termination criterion is not met}

	\State  \texttt{(S.1)} Pick $(i^k, \bd{d}^k)$;

    \State \texttt{(S.2)} Set:
	\begin{align*}
        \tau_{i^k j}^k = \max(\tau_{i^k j}^{k-1}, k-d_j^k), \quad \forall j \in \mathcal{N}_{i^k}^{\text{in}}.
        \end{align*}

      \State \texttt{(S.3)} Local Descent:
        \begin{align}\label{local-descent}
        \mathbf{v}_{i^k}^{k+1} = \mathbf{x}_{i^k}^k - \gamma^k\mathbf{z}_{i^k}^k.
        \end{align}
      \State \texttt{(S.4)} Consensus:
        \begin{align*}
        \mathbf{x}_{i^k}^{k+1} = w_{i^ki^k}\mathbf{v}_{i^k}^{k+1} + \sum_{j\in \mathcal{N}_{i^k}^{\text{in}}} w_{{i^k}j} \mathbf{v}_{j}^{\tau_{i^k j}^k}.
        \end{align*}
      \State \texttt{(S.5)} Gradient Tracking:	
         \begin{itemize}[leftmargin=1cm] \small 
     \item  \texttt{(S.5.1)} \textbf{Sum step:}
	\begin{align*} 
		 \bz_{i^k}^{k+\frac{1}{2}}  = & \bz_{i^k}^{k}  +  \displaystyle\sum_{j \in \mathcal{N}_{i^k}^{\text{in}}}  \left({\bm \rho}_{{i^k}j}^{\tau_{i^k j}^k }-\tilde{\bm \rho}_{i^kj}^k\right) \\
								    &	+ \nabla f_{i^k}(\bd{x}_{i^k}^{k+1}) - \nabla f_{i^k}(\bd{x}_{i^k}^{k})  \end{align*}
     \item \texttt{(S.5.2)} \textbf{Push step:}
		\begin{align*}	
		 & \bz_{i^k}^{k+1}    =  a_{i^ki^k}\, \bz_{i^k}^{k+\frac{1}{2}}  , \\
		& {\bm \rho}_{j {i^k}}^{k+1}   = {\bm \rho}_{j {i^k}}^{k} + a_{j {i^k}} \,\bz_{i^k}^{k+\frac{1}{2}}, \quad \forall j\in \mathcal{N}_{i^k}^{\text{out}}   \end{align*}
     \item \texttt{(S.5.3)} \textbf{Mass-Buffer update:}
		\begin{align*}	
		& \tilde{\bm \rho}_{{i^k}j}^{k+1} = {\bm \rho}_{{i^k}j}^{\tau_{i^k j}^k },   \label{eq:asy_ps_counter}  { \quad \,\ \forall j \in \mathcal{N}_{i^k}^{\text{in}}}
	\end{align*}
	\end{itemize}
	\State \texttt{(S.6)} \parbox[t]{\dimexpr\linewidth-\algorithmicindent}{ { Untouched state variables shift to state $k+1$ \\ while keeping the same value;    $k\leftarrow k+1$}.\strut}

    \EndWhile
  \end{algorithmic}\label{alg:AsyTracking}
\end{algorithm}

 In ASY-SONATA, agents continuously and with no coordination perform: i)  their local computations [cf. \textbf{{(S.3)}}], possibly using an out-of-sync estimate  $\bz_{i^k}^k$ of the average gradient; in   \eqref{local-descent}, $\gamma^k$ is  a step-size (to be properly chosen);
 ii) a  consensus step on the $\bx$-variables, using possibly outdated  information   $\bv_j^{\tau_{i^k j}^k}$ from their in-neighbors [cf. \textbf{{(S.4)}}]; and iii)  gradient tracking [cf. \textbf{{(S.5)}}] to update the local estimate $\bz_{i^k}^k$, based on  the current  cumulative mass variables  $  {\boldsymbol{\rho}}_{i^kj}^{\tau_{i^k j}^k}$,   and buffer variables $\tilde{{\bm \rho}}_{{i^k}j}^{k}$, $j \in \mathcal{N}_{i^k}^{\text{in}}$.   
 
Note that     in   Algorithm~\ref{alg:pertavg},  the tracking variable $\by_{i^k}^{k+1} $ is obtained rescaling $\bz_{i^k}^{k+1}$ by the factor $1/\phi_{i^k}^{k+1}$. In Algorithm~\ref{alg:AsyTracking}, we   absorbed  the scaling $1/\phi_{i^k}^{k+1}$  in the step size and use directly $\bz_{i^k}^{k+1}$ as a proxy of the average gradient, eliminating thus  the $\phi$-variables   (and the related  $\sigma$-, $\tilde{\sigma}$-variables). Also, for notational simplicity and without loss of generality, we assumed that the $\bv$- and ${\bm \rho}$- variables are subject to the same delays (e.g., they are transmitted within the same packet); same convergence results hold if  different delays are considered.
 
We study now   convergence   of the scheme, under  a constant step-size or   diminishing, uncoordinated  ones. \vspace{-0.3cm}

\subsection{Constant Step-size}\label{sec:const_step_size}

Theorem~\ref{thm:linear_const} below establishes {\it linear} convergence of ASY-SONATA when $F$ is strongly convex.  

\begin{theorem}[Geometric convergence]\label{thm:linear_const} 
Consider~(P) under Assumption ~\ref{ass:cost_functions_strongly_convex}, and let   $\bx^\star$ denote its unique solution. Let  {$\{(\bd{x}_i^k)_{i = 1}^ I\}_{k \in \mathbb{N}_0}$} be the sequence generated by Algorithm~\ref{alg:AsyTracking}, under Assumption~\ref{ass:strong_conn}, \ref{ass:delays}, and with weight-matrices   $\mathbf{W}$ and $\mathbf{A}$   satisfying Assumption~\ref{assumption_weights}.  Then, there exists a constant $\bar{\gamma}_1 >0$ [cf. \eqref{eq:gamma_1}] such that if   $\gamma^k \equiv \gamma \leq \bar{\gamma}_1$, it holds
\begin{align}
M_{\text{sc}}(\bx^k) \triangleq \| \bx^k - \bd{1}_I \otimes \bx^\star \|  = \mathcal{O}(\lambda^k) ,
\end{align}
with   $\lambda \in (0,1)$   given by 
\begin{equation}\label{eq:rate-expression}
\lambda = 
\begin{cases}
1 -  \frac{\tau \lbm^{2K_1} \gamma}{2}  &\text{if }  \gamma \in (0, \hat{\gamma}_1],\\
\rho + \sqrt{ J_1 \gamma} & \text{if } \gamma \in ( \hat{\gamma}_1, \hat{\gamma}_2) ,
\end{cases}
\end{equation}
where $\hat{\gamma}_1$ and  $\hat{\gamma}_2$ are some constants strictly smaller than $\bar{\gamma}_1$, and $J_1 \triangleq  (1-\rho)^2/\hat{\gamma}_2$.
\end{theorem}\vspace{-0.3cm}
\begin{proof}
	See Sec.~\ref{sec:theory}.\vspace{-0.1cm}
\end{proof}

When $F$ is convex (resp. nonconvex), we introduce the following merit  function   to measure the progresses of the algorithm towards optimality (resp.  stationarity) and consensus: 
\begin{equation}\label{eq:merit_function}
M_F(\bx^k) \triangleq \max \{\norm{\nabla F(\bar{\bd{x}}^k) }^2, \norm{ \bd{x}^k- \bd{1}_I \otimes \bar{\bd{x}}^k}^2\},
\end{equation}
where $\bd{x}^k \triangleq  [\bd{x}_1^{k\top}, \cdots, \bd{x}_I^{k\top}]^{\top}$ and $\bar{\bd{x}}^k \triangleq (1/I) \cdot \sum_{i=1}^I \bd{x}_i^k.$
Note that $M_F$ is a valid merit function,  since it is continuous and $M_F(\bx) =0$ if and only if all $\bx_i$'s are consensual and optimal (resp. stationary solutions).
\begin{theorem}[Sublinear convergence]\label{thm:sublinear_const}
 Consider~(P) under Assumption~\ref{ass:cost_functions} (thus possibly nonconvex). Let   {$\{(\bd{x}_i^k)_{i = 1}^I \}_{k \in \mathbb{N}_0}$} be the sequence generated by Algorithm~\ref{alg:AsyTracking}, in the same setting of Theorem~\ref{thm:linear_const}. Given $\delta>0$,  let $T_{\delta}$ be the first iteration $k\in \mathbb{N}_0$ such that $M_F (\bx^k)  \leq \delta$.
Then, there exists a  $\bar{\gamma}_2 >0$ [cf. \eqref{eq:gamma_2}], such that if $\gamma^k \equiv \gamma \leq \bar{\gamma}_2$,   {$T_\delta = \mathcal{O} (1/\delta)$.
The values of the above constants is given in the proof.\vspace{-0.1cm}
                                                                                    }
\end{theorem}
\begin{proof}
	See Sec.~\ref{sec:sublinear}.\vspace{-0.2cm}  
\end{proof}
Theorem~\ref{thm:linear_const} states that   consensus and optimization errors of the sequence generated by ASY-SONATA vanish  at a linear rate. 
  We are not aware of any other scheme enjoying such a property in  such  a  distributed, asynchronous computing environment. 
 For general, possibly nonconvex instances of  Problem~(P),  Theorem~\ref{thm:sublinear_const} shows that   both consensus and optimization errors of the sequence generated by ASY-SONATA vanish   at $\mathcal{O}(1/\delta)$ sublinear rate. 

The choice of a proper stepsize calls for the estimates  of   $\bar{\gamma}_1$ and  $\bar{\gamma}_2$ in Theorems \ref{thm:linear_const}   and \ref{thm:sublinear_const}, which depend on the following quantities: the optimization parameters  $L_i$ (Lipschitz constants of the gradients) and  $\tau$ (strongly convexity constant), the network connectivity parameter $\rho$, and the constants $D$ and $T$ due to the asynchrony (cf. Assumption \ref{ass:delays}). Notice that the  dependence of the stepsize on $L_i$,   $\tau$, and  $\rho$ is common to all the existing distributed synchronous algorithms   and so is that on $T$ and $D$ to (even centralized) asynchronous  algorithms \cite{Bertsekas_Book-Parallel-Comp}. While $L_i$,   $\tau$, and  $\rho$ can be acquired following approaches discussed in the literature (see, e.g., \cite[Remark 4]{Nedich-geometric}), it is less clear how to estimate $D$ and $T$, as they are related to the asynchronous  model, generally not known to the agents. As an example, we  address this question    considering the following    fairly  general model for the  agents' activations and asynchronous communications.   
 Suppose that the length of any time window between consecutive ``push'' steps of any agent belongs to  $[p_{\text{min}}, p_{\text{max}}]$, for some $p_{\text{max}}\geq p_{\text{min}}>0$,  and one agent always sends out its updated information immediately after the completion of its ``push'' step. 
 The traveling time of each packet is at most    $D^{\text{tv}}$. Also,  at least one packet is   successfully received  every $D^{\text{ls}}$ successive one-hop communications. Note that there is a vast literature on how to estimate $D^{\text{tv}}$ and $D^{\text{ls}}$, based upon the  specific channel model under consideration; see, e.g., \cite{Rappaport_book,Kay_book}.   In this setting, it is not difficult to check that one can set $T = (I-1)\ceil{{p_{\text{max}}}/{p_{\text{min}}}} + 1$ and $D = I \ceil{{D^{\text{tv}}}/{p_{\text{min}}}} D^{\text{ls}}.$
To cope with the issue of estimating    $\bar{\gamma}_1$ and  $\bar{\gamma}_2$, in the next section we show how to employ in    ASY-SONATA diminishing, uncoordinated stepsizes.

\subsection{Uncoordinated diminishing step-sizes}\label{sec:dimin_step_size}

The use of a diminishing  stepsize shared across the agents is quite common in synchronous distributed algorithms. However, it is not   clear how to implement such option in an  asynchronous setting, without enforcing any coordination among the agents (they should know  the global iteration counter $k$).   
 In this section, we provide for the first time a solution to this issue. Inspired by \cite{CanelliCAMSAP17}, our model assumes  that each agent, {\it independently} and with {\it no coordination} with the others, draws the step-size from a local   sequence  {$\{\alpha^t\}_{t\in \mathbb{N}_0}$}, according to its local clock. The sequence $\{\gamma^k\}_{k\in \mathbb{N}_0}$ in \eqref{local-descent} will be thus the result of the ``uncoordinated samplings'' of the local out-of-sync sequences $\{\alpha^t\}_{t\in \mathbb{N}_0}$.  
The next theorem shows that in this setting, ASY-SONATA converges at a sublinear rate for both convex and nonconvex objectives. 

\begin{theorem}\label{thm:dimi_sublinear}
  Consider Problem (P) under Assumption~\ref{ass:cost_functions} (thus possibly nonconvex). Let   {$\{(\bd{x}_i^k)_{i = 1}^I\}_{k \in \mathbb{N}_0}$} be the sequence generated by Algorithm~\ref{alg:AsyTracking}, in the same setting of Theorem~\ref{thm:linear_const}, but with the agents using a local step-size sequence   $\{\alpha^t\}_{t\in \mathbb{N}_0}$ satisfying   $\alpha^t \downarrow 0$ and   $\sum_{t=0}^\infty \alpha^t =\infty$. Given $\delta>0$, let $T_{\delta}$ be the first iteration $k\in \mathbb{N}_0$ such that $M_F (\bx^k)  \leq \delta$. Then \vspace{-0.1cm}
 \begin{equation}\label{eq:sublinear-rate}
 T_\delta \leq \inf  \Big\{ k \in \mathbb{N}_0 \,\Big\vert \,\sum_{t=0}^{k} \gamma^t \geq c/\delta \Big\},\vspace{-0.1cm}\end{equation}
where  $c$ is a positive constant. \vspace{-0.1cm}
\end{theorem}
\begin{proof}
See Sec. \ref{sec:sublinear}.	\vspace{-0.2cm}
\end{proof}

\section{Numerical Results}

We test \ASYSONATA/ on the least square regression and the binary classification problems.  The MATLAB code can be found at \url{https://github.com/YeTian-93/ASY-SONATA}.
\label{sec:simulation} 

\subsection{Least square regression}
In the LS problem,  each agent $i$ aims to estimate an unknown signal $\mathbf{x}_0 \in \mathbb{R}^n$ through   linear measurements   $\mathbf{b}_i = \mathbf{M}_i\mathbf{x}_0+\mathbf{n}_i$, where  $\mathbf{M}_i \in \mathbb{R}^{ d_i \times n}$ is the sensing matrix,  and $\mathbf{n}_i \in \mathbb{R}^{d_i}$ is the additive  noise.  The  LS problem can be written in the form of (P), with  each $f_i(\mathbf{x})= \|\mathbf{M}_i\mathbf{x}-\mathbf{b}_i\|^2$.  

\textbf{Data:} We fix $\bx_0$ with its elements being  i.i.d. random variables  drawn from the  standard normal distribution.  For each $\mathbf{M}_i$, we firstly generate all its elements as i.i.d. random variables drawn from the standard normal distribution, and then normalize the matrix by multiplying it with the reciprocal of its spectral norm.  The elements of the additive  noise $\mathbf{n}_i$  are i.i.d. Gaussian distributed, with zero mean   and variance equal to $0.04$.  We set $n = 200$ and $d_i = 30$ for each agent.
\textbf{Network model:} We simulate a network of $I = 30$ agents.  Each agent $i$ has $3$ out-neighbors; one of them belongs to a directed cycle graph connecting all the agents while the other two are picked uniformly at random. 
\textbf{Asynchronous model:} Agents are activated according to a cyclic rule where the order is randomly permuted at the beginning of each round.  Once activated,  every agent  performs  all the steps as in Algorithm \ref{alg:AsyTracking} and then sends  its updates  to all its out-neighbors. Each transmitted message  has  (integer) traveling time which is drawn uniformly at random within the interval $[0, D^{\text{tv}}]$.  We set $D^{\text{tv}} = 40.$  

\begin{figure}[t] 
 \centering{\includegraphics[scale=0.35]{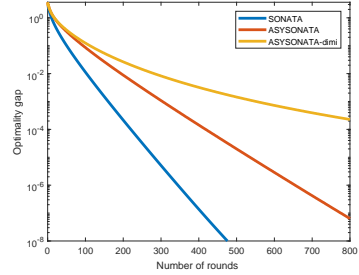}}
  \caption{\small Directed graphs: optimality gap $J^k$ versus number of rounds.}
  \label{fig:LS_dGraph}  \vspace{-.5cm}
\end{figure}

We test \ASYSONATA/   with a constant step size $\gamma = 3.5$, and also a diminishing step-size rule with each agent updating its local step size according to $\alpha^{t+1} = \alpha^{t} \left(1- 0.001\cdot \alpha^{t} \right)$ and $\alpha^0 = 3.5$; as   benchmark, we also simulate  its synchronous instance, with step size  $\gamma = 0.8$.   In Fig.~\ref{fig:LS_dGraph},  we plot $J^k \triangleq  ({1}/{I})\,\sqrt{\sum_{i=1}^I \| \mathbf{x}_i^k - \mathbf{x}^\star\|_2^2}$    versus the number of rounds (one round corresponds to one update of all the agents). The curves are averaged over $100$ Monte-Carlo simulations, with different graph and data instantiations. The plot clearly shows linear convergence of \ASYSONATA/with a constant step-size.  

\subsection{Binary classification}
In this subsection, we consider a strongly convex and nonconvex instance of  Problem~(P) over  digraphs, namely:  the regularized logistic regression    (RLR) and   the robust classification (RC) problems. 
 Both formulations  can be abstracted as:\vspace{-0.2cm} \begin{equation}\label{eq:P_sim}
\min_{\mathbf{x}} \frac{1}{\abs{\mathcal{D}}} \sum_{i=1}^I \sum_{j\in \mathcal{D}_i} V(y_j \cdot \ell_{ \mathbf{x}}(\mathbf{u}_j)) + \lambda \norm{\nabla \ell_{ \mathbf{x}}(\cdot)}_2^2,\vspace{-0.1cm} \end{equation} where   $\mathcal{D} = \cup_{i=1}^I \mathcal{D}_i$ is the set of indices of the data distributed across the agents, with agent $i$ owning  $\mathcal{D}_i$, and $\mathcal{D}_i \cap \mathcal{D}_l = \emptyset,$  for all $i \neq l$; $\mathbf{u}_j$ and $y_j\in \{-1,1\}$ are the feature vector and associated label of the $j$-th sample in $\mathcal{D}$;     $\ell_{ \mathbf{x}}(\cdot)$ is a linear function, parameterized by $\mathbf{x}$; and   $V$ is the loss function. More specifically,  if the RLR problem is considered, $V$    reads   $V(r) = \frac{1}{1+e^{-r}}$  while for the   RC problem,  we have \cite{zhao2010convex}\vspace{-0.2cm}
\begin{align*}
V(r) = 
\begin{cases}
0,  & \text{if }  r>1  ;  \\
\frac{1}{4} r^3 - \frac{3}{4} r + \frac{1}{2}, &  \text{if }  -1 \leq r \leq 1 ; \\
1,  & \text{if }  r<-1. 
\end{cases}
\end{align*} 
\noindent\textbf{Data:}  We use the following data sets for  the RLR and RC problems.  {(RLR):} We set $\ell_{ \mathbf{x}}(\mathbf{u}) = \mathbf{x}^\top \mathbf{u}$, $n={100}$, each $\abs{\mathcal{D}_i}=20,$ and $\lambda = 0.01.$  The underlying statistical model is the following:  We  generated the ground truth $\widehat{\mathbf{x}}$ with i.i.d.  $\mathcal{N}(0,1)$ components;  each training  pair $(\mathbf{u}_j,y_j)$ is generated independently, with  each element of $\mathbf{u}_j$ being  i.i.d.  $\mathcal{N}(0,1)$ and $y_j$ is set as $1$ with probability $V(\ell_{\widehat{\bx}}(\mathbf{u}_j)),$   and $-1$ otherwise. {(RC):}  We use the Cleveland Heart Disease Data set with 14 features \cite{Dua:2019}, preprocessing it by deleting observations with missing entries, scaling features between 0-1, and distributing the data to agents evenly.  We set $\ell_{ \mathbf{x}}(\mathbf{u}) = \mathbf{e}_{15}^\top \mathbf{x}  + \sum_{d=1}^{14}\mathbf{e}_d^\top \mathbf{x} \,\mathbf{e}_d^\top \mathbf{u}.$  
\textbf{Network model:}  We simulated a digraph of $I = 30$ agents.  Each agent has $7$ out-neighbors; one of them belongs to a directed cycle connecting all the agents while the other $6$ are picked uniformly at random.  One row and one column stochastic matrix with  uniform weights are generated. 
\textbf{Asynchronous model:}  a) Activation lists are generated by concatenating {\it random rounds.}  To generate one round, we first sample its length uniformly from the interval $[I, T]$, with  $T=90.$  Within a  round, we first have each agent appearing exactly once and then sample agents uniformly for the remaining spots.  Finally a random shuffle of the agents order is performed on each round; b) Each transmitted message  has  (integer) traveling time which is sampled uniformly from the interval $[0, D^{\text{tv}}]$, with $D^{\text{tv}}=90$.

We compare the performance of our algorithm with AsySubPush \cite{assran2018asynchronous} and AsySPA \cite{zhang2018asyspa}, which   appeared online  during the revision process of our paper.  AsySubPush   and AsySPA differ from \ASYSONATA/ in the following aspects: i) they do  not employ any gradient tracking mechanism; 
 ii) they cannot handle packet losses and purge out old information from the system (information is used as it is received); iii) when $F$ is strongly convex, they provably converge at {\it sublinear} rate; and iv) they cannot handle nonconvex $F$.   
The step sizes of all algorithms are manually tuned to obtain the best practical performance.  We run two instances of  \ASYSONATA/, one employing    a constant step size $\gamma=0.4$ and the other one using   the diminishing step size rule $\alpha^{t+1} = \alpha^{t} \left(1- 0.001\cdot \alpha^{t} \right)$, where $\alpha^0 = 0.5$ and  $t$ is the local iteration counter.  For AsySubPush (resp. AsySPA)  we set, for each agent $i$, $\alpha_i=0.0001$ (resp. $\rho(k)= {c}/{\sqrt{k}}$ with $c=0.01$) in RLC and $\alpha_i=0.00001$ (resp. $\rho(k)={c}/{\sqrt{k}}$ with $c=0.001$) in RC. 
The result is averaged over 20 Monte Carlo experiments with different digraph instances, and is presented in Fig.~\ref{fig:simul};   for each algorithm, we plot the merit functions $M_{\text{sc}}$ (left panel) and $M_{\text{F}}$ (right panel) evaluated in the generated trajectory   versus the global iteration counter $k$.   Consistently with the convergence theory, \ASYSONATA/ with a constant step size exhibits a linear convergence rate. Also, \ASYSONATA/ outperforms the other two algorithms; this is mainly due to i) the presence in  \ASYSONATA/    of an asynchronous gradient tracking mechanism which provides, at each iteration,  a better estimate of   $\nabla F$; and ii) the possibility in \ASYSONATA/ to discard old information  when received after a newer one [cf. (\ref{eq:purge_old})].
 
\begin{figure}[t!] 
 \centering{\includegraphics[scale=0.35]{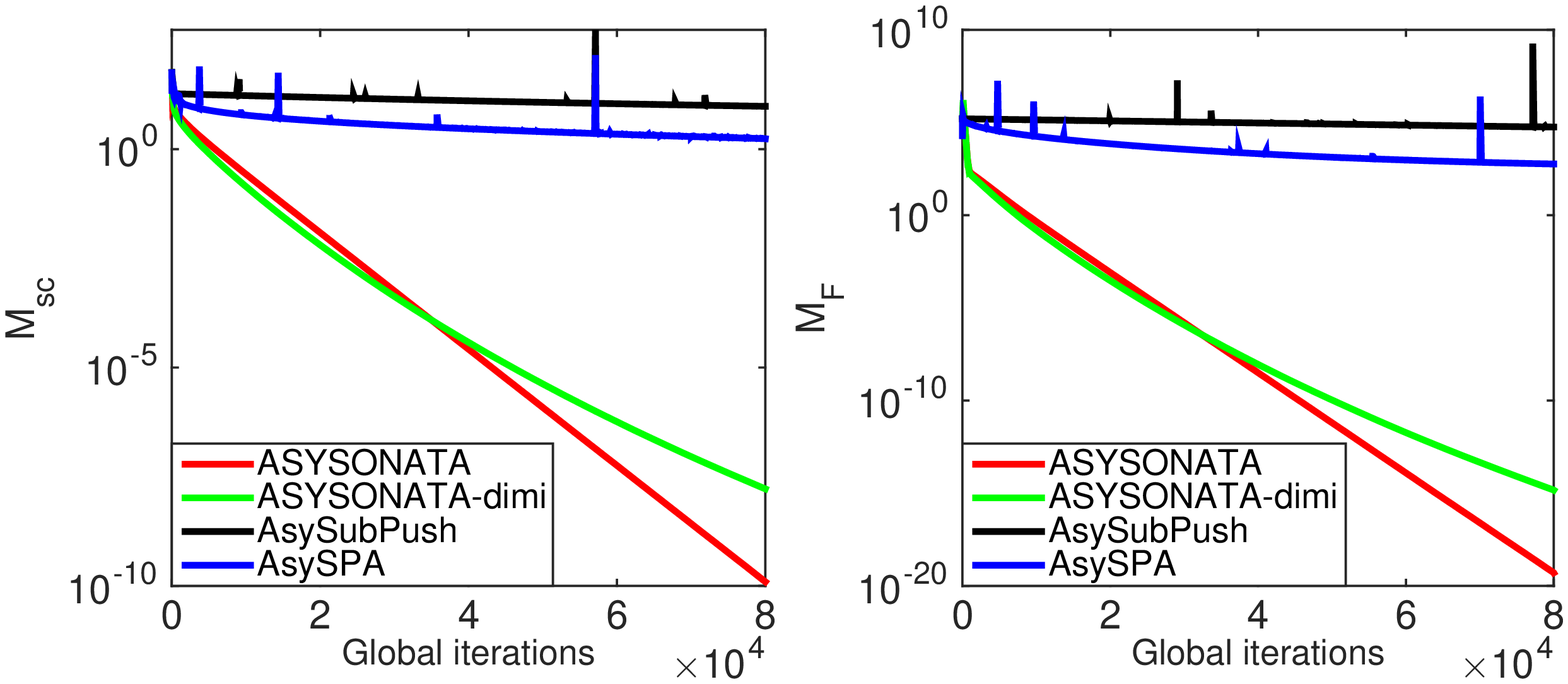}}\vspace{-.3cm}
  \caption{\small L: regularized logistic regression; R: robust classification.}
  \label{fig:simul}  \vspace{-.1cm}
\end{figure}

\section{Convergence Analysis of  \PertAvgname/}\label{sec:pertavg}
 
We prove Theorem \ref{thm:track}; we
 assume   $n=1$, without loss of generality. 
The proof   is organized in the following two steps. \textbf{Step 1:} We first reduce  the asynchronous agent system   to a synchronous ``augmented'' one with no delays. This will be done adding  virtual agents to the graph $\mathcal{G}$ along with their state variables, so that    \PertAvgname/ will be rewritten   as a  (synchronous) perturbed    push-sum algorithm   on the augmented graph.  While this idea was first explored in  \cite{dominguez2011distributed1,nedic2010convergence},   there are some important differences between the proposed enlarged systems and those used therein, see   Remark \ref{rmk_augmented_graph}. \textbf{Step 2:} We conclude the proof   establishing convergence of the   perturbed    push-sum algorithm built in Step 1. \vspace{-0.3cm}

\subsection{Step 1:  Reduction to a synchronous perturbed push-sum}


\subsubsection{The augmented graph}
We begin constructing the  augmented graph--an enlarged agent system  obtained adding virtual agents  to the original graph  $\mathcal{G}=(\mathcal{V},\mathcal{E})$.   Specifically, we associate to   each   edge  $(j,i) \in \mathcal{E}$  an ordered set of virtual nodes (agents), one for each of the possible delay values, denoted with a slight abuse of notation by $(j,i)^0,(j,i)^1,\ldots, (j,i)^D$; see Fig.~\ref{fig:original}.   Roughly speaking, these virtual nodes store the ``information on fly'' based upon  its associated delay, that is, the information that has been generated by  $j\in  \mathcal{N}^{\text{in}}_i$  for $i$  but not used (received) by $i$ yet. Adopting the terminology in \cite{nedic2010convergence},   nodes in the original graph  $\mathcal{G}$ are termed {\it computing agents} while the virtual nodes will be called  {\it noncomputing agents}. 
 With a slight abuse of notation, we define the set of computing and noncomputing  agents   as $\widehat{\mathcal{V}} \triangleq \mathcal{V} \cup \{(i,j)^d \lvert \, (i,j) \in \mathcal{E},\, d = 0,1,\ldots, D\} $, and its cardinality as $S \triangleq \abs{\widehat{\mathcal{V}}} = \left( I+(D+1) \abs{\mathcal{E}} \right)$.  We now identify the neighbors of each agent in this augmented systems. Computing agents no longer communicate among themselves; each $j\in \mathcal{V}$ 
 can only send information to the noncomputing nodes $(j,i)^0$, with $i\in \mathcal{N}^{\text{out}}_j$. Each noncomputing agent  $(j,i)^d$ can either  send information to the next noncomputing agent, that is  $(j,i)^{d+1}$ (if any), or to the computing agent  $i$; see Fig.~\ref{fig:original}(b).\begin{figure}
\centering
\includegraphics[scale=0.3]{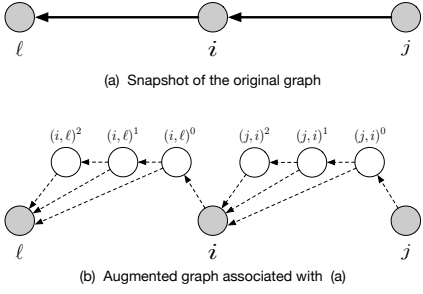}\vspace{-0.3cm}
\caption{Example of augmented graph, when the maximum delay $D=2$;  three noncomputing agents  are added for each edge $(j,i)\in \mathcal{E}$.}
\label{fig:original}\vspace{-0.5cm}
\end{figure}

To  describe the information stored by the agents in the augmented system at each iteration, let us first   introduce the following quantities:   $\mathcal{T}_i \triangleq \left\{ k  \, \big\vert \, i^{k} =i   , \, k\in \mathbb{N}_0\right\}$ is the set of global iteration indices at which the computing agent $i\in \mathcal{V}$ wakes up;  and, given $k\in \mathbb{N}_0$, let  $\mathcal{T}_i^{k} \triangleq \left\{ t\in \mathcal{T}_i \, \big\vert \, t \leq k  \right\}$.  
It  is not difficult to conclude from \eqref{eq:alg_rho} and \eqref{eq:alg_buffer} that  \vspace{-.1cm}
\begin{equation}\label{eq:rho_z}
 \hspace{-0.5cm}\rho_{ij}^{k}  =  \!\sum_{t \in \mathcal{T}^{k-1}_j} a_{ij}z^{t + 1/2}_{j}  \,\,\,\text{and}\,\,\,
 \tilde{\rho}^{k}_{ij} =  \!\rho_{ij}^{\tau_{ij}^{k-1}},\quad   (j,i)\in \mathcal{E}.\vspace{-.3cm}
\end{equation} 
At iteration $k=0$, every computing agent  $i$ stores $z_i^0$, whereas the values of the noncomputing agents are initialized to $0.$ 
At the beginning of iteration $k$, every computing agent  $i$  will store   $z_i^k$ whereas  every noncomputing agent $(j,i)^d$, with $0\leq d \leq D-1$, stores  the mass $a_{ij}z_j$  (if any)   generated by $j$ for $i$ at iteration  $k-d-1$ (thus $k-d-1 \in \mathcal{T}_j^{k-1}$), i.e., $a_{ij}z_j^{k-(d+1)+1/2}$ (cf. Step 3.2), and not been used by  $i$ yet (thus $k-d > \tau_{ij}^{k-1}$); otherwise it stores $0$. Formally, we have
				\begin{align} \hspace{-0.2cm}  z^k_{(j,i)^d} \triangleq  & a_{ij}z_j^{t + 1/2 } \nonumber \\
				\label{eq:virtual_info_d}			& \cdot \mathbbm{1}{\left[ t = k-d-1 \in \mathcal{T}_j^{k-1}\,\,  \& \,\, t +1 > \tau_{ij}^{k-1} \right].}  \end{align}
The virtual node $(j,i)^D$ cumulates  all the masses $a_{ij}z_j^{k-(d+1)+1/2}$ with   $d \geq D$, not  received by $i$ yet: 
\begin{equation}\label{eq:virtual_info}
\begin{aligned}
&   z^k_{(j,i)^D} \triangleq \sum_{t \in  \mathcal{T}_j^{k-D-1},\, t +1 > \tau_{ij}^{k-1}} a_{ij}z_j^{t + 1/2}.
\end{aligned}
\end{equation}
We   write next \PertAvgname/ on the augmented graph in terms of the $z$-variables of both the computing and noncomputing agents, absorbing the $(\rho, \tilde{\rho})$-variables using \eqref{eq:rho_z}-\eqref{eq:virtual_info}. 

\noindent\textbf{The sum-step over the augmented graph.}
 In the sum-step, the update of the $z$-variables of the computing agents reads:
 \vspace{-0.1cm}\begin{subequations} \label{eq:pf_sum}
\begin{align} 
   & z_{i^k}^{k+\frac{1}{2}}    =   z_{i^k}^{k}  +  \displaystyle \sum_{j \in \mathcal{N}_{i^k}^{\text{in}}}  \left({\rho}^{\tau_{i^k j}^k}_{i^k j}-\tilde{\rho}_{i^k j}^k\right) + \epsilon^k \vspace{-0.2cm}\nonumber\\ 
 &\,\,\quad  \overset{\eqref{eq:rho_z}-\eqref{eq:virtual_info}}{=}  z_{i^k}^{k}  +  \displaystyle \sum_{j \in \mathcal{N}_{i^k}^{\text{in}}} \sum_{d=k - \tau_{i^k j}^k}^D z^k_{(j,i^k)^d}+ \epsilon^k; \label{eq:pf_sum_1}
 \end{align}
\begin{equation} 
   z^{k+\frac{1}{2}}_{j} = z^{k }_{j},  \quad  j \in \mathcal{V}\setminus \{i^k\} .\label{eq:pf_sum_2}\vspace{-0.2cm}
 \end{equation}
 In words, node $i^k$ builds the update $z_{i^k}^{k} \!\!\to\!\! z_{i^k}^{k+\frac{1}{2}}$ based upon the masses transmitted by the noncomputing agents    
 $(j,i^k)^{k - \tau_{i^k j}^k},(j,i^k)^{k - \tau_{i^k j}^k+1},\ldots,$ $(j,i^k)^D$ [cf. \eqref{eq:pf_sum_1}].  All the other computing agents keep their masses unchanged [cf. \eqref{eq:pf_sum_2}].  The  updates of the noncomputing agents is set to
 \begin{align}
 & z^{k+\frac{1}{2}}_{(j,i^k)^d} \triangleq 0,  \quad   d=k - \tau_{i^k j}^k,\ldots,D,\quad   j \in \mathcal{N}_{i^k}^{\text{in}};\label{eq:pf_sum_3}\smallskip \\
& z^{k+\frac{1}{2}}_{(j^\prime,i)^\tau} \triangleq z^{k}_{(j^\prime,i)^\tau}, \quad   \text{for all the other }   (j^\prime,i)^\tau\in \widehat{\mathcal{V}}.\label{eq:pf_sum_4}
\vspace{-0.4cm}
\end{align} \end{subequations}
The  noncomputing agents in  \eqref{eq:pf_sum_3}    set  their    variables to zero (as they transferred their  masses to   $i^k$) while  the  other   noncomputing agents  keep their   variables unchanged [cf.\,\eqref{eq:pf_sum_4}].  Fig. \ref{fig:sum} illustrates the sum-step over the augmented graph.   
\begin{figure}[t]
\centering
\includegraphics[scale=0.3]{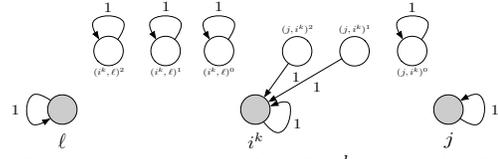} \vspace{-0.3cm}
\caption{Sum step on the augmented graph:  $\tau_{i^k j}^k = k-1$ (delay one); the two noncomputing agents, $(j,i^k)^1$ and $(j,i^k)^2$,  send  their  masses to $i^k.$}
\label{fig:sum}\vspace{-0.3cm}
\end{figure}
\begin{figure} 
\centering
\includegraphics[scale=0.3]{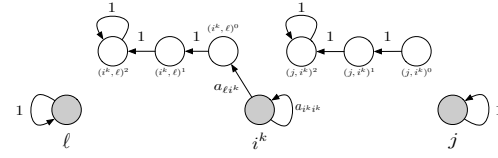}\vspace{-0.2cm}
\caption{Push step on the augmented graph:  Agent $i^k$   keeps $a_{i^ki^k}z_{i^k}^{k+1/2}$   while sending $a_{\ell i^k}z_{i^k}^{k+1/2}$ to the virtual nodes $(i^k,\ell)^0$,   $\ell \in \mathcal{N}_{i^k}^{\text{out}}$.}\vspace{-0.3cm}
\label{fig:push}
\end{figure}

\noindent\textbf{The push-step over the augmented graph.}
 In the push-step, the update of the $z$-variables of the computing agents reads:\vspace{-0.2cm}
 \begin{subequations}\label{eq:pf_push}
\begin{align}
   &z_{i^k}^{k+1} = a_{i^k i^k}\,z_{i^k}^{k+\frac{1}{2}}; \label{eq:pf_push_1}\\
    & z^{k+1}_{j} = z^{k+ \frac{1}{2}}_{j}, \qquad \text{ for } j \in \mathcal{V}\setminus \{i^k\}. \label{eq:pf_push_2}
\end{align}
In words, agent $i^k$ keeps the portion  $a_{i^k i^k} {z}_{i^k}^{k+ \frac{1}{2}}$  of the new generated mass    [cf. \eqref{eq:pf_push_1}] whereas  the other computing agents  do not change their variables [cf. \eqref{eq:pf_push_2}].  The  noncomputing agents update as:   \begin{align}
& z^{k+1}_{(i^k,\ell)^{0}} \triangleq  a_{\ell i^k}\, z_{i^k}^{k+1/2}, \quad   \ell \in \mathcal{N}_{i^k}^{\text{out}}; \label{eq:pf_push_5}\\
& z^{k+1}_{(i,j)^{0}} \triangleq 0, \quad  (i,j) \in \mathcal{E},\quad   i \neq i^k; \label{eq:pf_push_6}\\
& z^{k+1}_{(i,j)^{d}} \triangleq z^{k+\frac{1}{2}}_{(i,j)^{d-1}},  \quad d= 1 ,\ldots, D-1, \quad   (i,j) \in \mathcal{E}; \label{eq:pf_push_4} \\
& z^{k+1}_{(i,j)^{D}} \triangleq  z^{k+\frac{1}{2}}_{(i,j)^{D}} + z^{k+\frac{1}{2}}_{(i,j)^{D-1}}, \quad   (i,j) \in \mathcal{E}. \label{eq:pf_push_3} 
 \end{align}
 \end{subequations}
 In words, the computing agent $i^k$ pushes its masses  $a_{\ell i^k} {z}_{i^k}^{k+\frac{1}{2}}$    to the noncomputing agents $(i^k, \ell)^0$, with   $\ell \in \mathcal{N}_{i^k}^{\text{out}}$ [cf. \eqref{eq:pf_push_5}]. As the other noncomputing agents $(i,j)^{0}$,   $i \neq i^k$, do not receive any mass for their associated computing agents, they set their variables to zero [cf. \eqref{eq:pf_push_6}]. Finally the other  noncomputing agents $(i,j)^d$, with $0\leq d \leq D-1$, transfers  their mass to the next  noncomputing node $(j,i)^{d+1}$ [cf. \eqref{eq:pf_push_3}, \eqref{eq:pf_push_4}]. This push-step   is  illustrated in Fig. \ref{fig:push}. 
 
 The following result establishes   the equivalence between the update of the enlarged system with that of Algorithm~\ref{alg:pertavg}.

\begin{proposition}\label{prop_equivalence} Consider the setting of Theorem \ref{thm:track}.  The values of the $z$-variables of the computing agents in \eqref{eq:pf_sum}-\eqref{eq:pf_push}   coincide with those of the $z$-variables generated by \PertAvgname/  (Algorithm~\ref{alg:pertavg}), for all iterations $k\in \mathbb{N}_0$.  \end{proposition}
\begin{proof}
By construction, the updates of the computing agents as in \eqref{eq:pf_sum_1}-\eqref{eq:pf_sum_2} and \eqref{eq:pf_push_1}-\eqref{eq:pf_push_2}   coincide with the z-updates in the sum- and push-steps of  \PertAvgname/, respectively. Therefore, we only need to show that the updates of the noncomputing agents  are consistent with those of the $(\rho,\tilde{\rho})$-variables in \PertAvgname/.  This follows   using \eqref{eq:rho_z} and noting that the updates  \eqref{eq:pf_push_5}-\eqref{eq:pf_push_3}  are compliant   with  \eqref{eq:virtual_info_d} and \eqref{eq:virtual_info}. For instance, by \eqref{eq:rho_z}-\eqref{eq:virtual_info_d}, it must be $
 z^{k+1}_{(i^k,j)^0} =   a_{ji^k}z_j^{t + 1/2 }   \cdot \mathbbm{1}[ t =k \in \mathcal{T}_{i^k}^{k} \text{ and } t +1 > \tau_{ji^k}^{k}]  
 = a_{ji^k}z_j^{k + 1/2 }$, which in fact coincides with \eqref{eq:pf_push_5}. 
The other equations \eqref{eq:pf_push_6}--\eqref{eq:pf_push_3} can be  similarly validated. 
\end{proof}

Proposition \ref{prop_equivalence} opens the way to study convergence of 
\PertAvgname/ via that of the synchronous perturbed push-sum algorithm \eqref{eq:pf_sum}-\eqref{eq:pf_push}. To do so, it is convenient to  rewrite  \eqref{eq:pf_sum}-\eqref{eq:pf_push}  in vector-matrix form, as described next.
 
We begin introducing an enumeration rule for the components of the z-vector in the augmented system. 
 We enumerate all the elements of $\mathcal{E}$ as $1,2,\ldots, \abs{\mathcal{E}}.$  The computing agents in   $\widehat{\mathcal{V}}$  are   indexed as in $\mathcal{V}$, that is, $1,2,\ldots, I$.   {Each noncomputing agent $(j,i)^d$ is indexed as $I+ d\abs{\mathcal{E}} + s$, where  $s$ is the index associated with  $(j,i)$ in $\mathcal{E}$;   we will use interchangeably $z_{I+ d\abs{\mathcal{E}} + s}$ and $z_{(j,i)^d}.$  We define the $z$-vector as $\fz = [{z}_i]_{i=1}^S$; and its value at iteration $k\in \mathbb{N}_0$ is denoted by $\fz^k$. }


The transition matrix $\C^k$ of the sum step is defined as
\[
\eC_{hm}^k  \triangleq 
\left\{
\begin{aligned}
& 1 ,  && \text{if }  m \in \{ (j, i^k)^{d} \mid  k - \tau_{i^k j}^k \leq d \leq D \} \\
&        && \text{and }  h = i^k; \\
& 1 ,  && \text{if }  m \in \widehat{\mathcal{V}} \setminus \{ (j, i^k)^{d} \mid  k - \tau_{i^k j}^k \leq d \leq D \} \\
&        && \text{and }  h = m; \\
& 0,   && \text{otherwise}.
\end{aligned}
\right.
\]
Let  $\perty^k \triangleq \epsilon^k \mathbf{e}_{i^k}$ be the  $S-$dimensional perturbation vector. The sum-step can be written in compact form as \vspace{-0.1cm}
\begin{equation}\label{eq:sum_matrix}
\fz^{k+\frac{1}{2}} = \C^k \fz^k + \perty^k.\vspace{-0.1cm}
\end{equation}

Define the transition matrix $\Sd^k$ of the push step as
\[
\eSd_{hm}^k  \triangleq 
\left\{
\begin{aligned}
& a_{j i^k} ,  && \text{if }  m = i^k \text{ and } h = (j,i^k)^0, j\in \mathcal{N}_{i^k}^{\text{out}} ; \\
& a_{i^k i^k} ,  && \text{if }  m =h= i^k; \\
& 1 ,  && \text{if }  m =h\in \mathcal{V} \setminus i^k; \\
&  1,      && \text{if }  m = (i,j)^d,\, h = (i,j)^{d+1}, \\
&        && (i,j) \in \mathcal{E}, \, 0 \leq d \leq D-1; \\
&  1,      && \text{if }  m =h= (i,j)^D,\,  (i,j) \in \mathcal{E}; \\
& 0,   && \text{otherwise}
\end{aligned}
\right.
\]
Then, the push-step can be written as 
\begin{align}\label{eq:push_matrix}
\fz^{k+1} = \Sd^k \fz^{k+\frac{1}{2}}.
\end{align}
Combing \eqref{eq:sum_matrix} and \eqref{eq:push_matrix}, yields 
\begin{align}\label{eq:tracking_dynamics-z}
\fz^{k+1} = \M^k \fz^k + {\gerror}^k,\quad   \M^k \triangleq \Sd^k \C^k,\quad {\gerror}^k \triangleq \Sd^k \perty^k.
\end{align}

The updates of  the $\phi$ variables and the definition of the $\phi$-vector are similar as above.   In summary,  the \PertAvgname/ algorithm can be rewritten in   compact form as\begin{subequations}\label{eq:augmented_consensus}
 \begin{align} 
&
 {\fz}^{k+1} = {\M} ^k {\fz}^{k}+{\gerror}^k,\quad  {\gerror}^k=   \epsilon^k\,\Sd^k  \,\mathbf{e}_{i^k};\label{eq:ordy_1} \\
& {\fm}^{k+1} = {\M}^k {\fm}^k;\label{eq:ordy_2}
\end{align}\end{subequations}
 with initialization: $z^{0}_i  \in \mathbb{R}$ and $\phi^{0}_i = 1$, for  $ i \in  \mathcal{V}$; and $z^{0}_i = 0$ and $\phi^{0}_i = 0$, for $i \in  \widehat{\mathcal{V}} \setminus \mathcal{V}$.\vspace{-0.2cm}

\begin{remark}[Comparison with \cite{dominguez2011distributed1,nedic2010convergence,lin2013constrained,BofCarli17}]\label{rmk_augmented_graph} 
The idea of reducing asynchronous (consensus) algorithms into synchronous ones over an   augmented  system   was already explored in \cite{nedic2010convergence,lin2013constrained,BofCarli17}.  However, there are several important differences between the models   therein and the proposed augmented graph.  First of all, \cite{BofCarli17} extends the analysis in \cite{dominguez2011distributed1} to deal with asynchronous activations, but both work consider only packet losses (no  delays). Second, our augmented graph model departs from that in \cite{nedic2010convergence,lin2013constrained} in the following aspects: i) in our model, the virtual nodes are associated with the {\it edges} of the original graph rather than the nodes;  ii) the noncomputing nodes store the  information {\it on fly} (i.e., generated by a sender but   not received by the intended receiver yet), while in \cite{nedic2010convergence,lin2013constrained}, each noncomputing agent owns a delayed copy of  the message generated by the associated computing agent; and iii) the dynamics \eqref{eq:augmented_consensus} over the augmented graph used to describe the P-ASY-SUM-PUSH procedure is different from those of the asynchronous consensus schemes \cite[(1)]{nedic2010convergence} and  \cite[(1)]{lin2013constrained}.
\end{remark}

 \subsection{Step 2: Proof of Theorem~\ref{thm:track}}
\subsubsection{Preliminaries}   
We begin studying some    properties of the matrix product $\M^{k:t}$, which will be instrumental to prove convergence of   the perturbed push-sum scheme \eqref{eq:augmented_consensus}.\vspace{-0.1cm} 

\begin{lemma}\label{lm:scramb}
Let $\{\M^k \}_{k\in \mathbb{N}_0}$ be the sequence of matrices in \eqref{eq:augmented_consensus},   generated by Algorithm~\ref{alg:pertavg}, under Assumption~\ref{ass:delays}, and with   $\mathbf{A}\triangleq (a_{ij})_{i,j=1}^I$ satisfying Assumption~\ref{assumption_weights} (i),(iii). The following hold: for all  $k\in \mathbb{N}_0$,  a)  $\M^k$ is column stochastic; and b)   the  entries of the  first $I$ rows  of  $\M^{k+K_1-1:k}$ are uniformly lower bounded  by $\eta \triangleq \lbm^{K_1}\in (0,1)$,    with $K_1 \triangleq (2I-1) \cdot T+I \cdot D $.\vspace{-0.1cm}
\end{lemma}
\begin{proof}
The lemma essentially proves that    $(\M^{k+K_1-1:k})^\top$ is a SIA (Stochastic Indecomposable Aperiodic) matrix  \cite{lin2013constrained}, by showing that  for any time length of $K_1$  iterations,
there exists a path from any node $m$ in the augmented graph to any computing node $h.$ While at a high level the proof shares some  similarities with   that of \cite[Lemma 2]{nedic2010convergence} and \cite[Lemma 5 (a)]{lin2013constrained}, there are important  differences due to the distinct modeling of our augmented system.  The complete proof is in Appendix~\ref{pf:Mscramb}.
\end{proof}

 The key result of this section is stated next and shows  that as $k-t$ increases,  $\widehat{\bA}^{k:t}$ approaches a column stochastic rank one matrix at a   linear rate.  Given  Lemma~\ref{lm:scramb}, the proof  follows the path of   \cite[Lemma~4, Lemma~5]{nedic2010convergence}, \cite[Lemma 4, Lemma 5(b,c)]{lin2013constrained} and thus is omitted.
 
  \begin{lemma}\label{lm:rate1} In the setting above, 
there exists  a sequence of stochastic vectors $\{{\bm{\xi}}^k\}_{k\in \mathbb{N}_0}$ such that,    for any $ k \geq t\in \mathbb{N}_0$ and   $i,\, j\in \{1, \cdots, S\}$, there holds\vspace{-0.1cm}
\begin{equation}\label{eq:Lemma_exp_decay}\abs{{\eM} ^{k:t}_{ij} -\xi^k_i} \leq C \rho ^{k-t}, \vspace{-0.2cm}\end{equation}
with\vspace{-0.2cm} $$ C \triangleq 2\frac{1+\lbm^{-K_1}}{1-\lbm^{K_1}}, \quad   \rho\triangleq (1-\lbm^{K_1})^{\frac{1}{K_1}}\in (0,1).$$  Furthermore, $\xi_i^k\geq \lbpsi$, for all  $i\in \mathcal{V}$ and $k\in \mathbb{N}_0$. 
\end{lemma} 

 \subsubsection{Proof of Theorem~\ref{thm:track}}  
 Applying \eqref{eq:augmented_consensus}  telescopically,  yields: 
${\fz}^{k+1} = {\M} ^{k:0} {\fz}^{0}+\sum_{l=1}^k {\M} ^{k:l} {\gerror}^{l-1}+{\gerror}^k$ and  ${\fm}^{k+1} = {\M}^{k:0} {\fm}^0,$
which using the column  stochasticity of $\M^{k:t}$, yields\vspace{-0.2cm}
\begin{align}
\label{eq:z_recur} 
\hspace{-0.4cm}\bd{1}^{\top} {\fz}^{k+1} = \bd{1}^{\top} {\fz}^{0}+\sum_{l=0}^k \bd{1}^{\top}  {\gerror}^{l}   ,\quad \bd{1}^{\top} {\fm}^{k+1}  = \bd{1}^{\top} {\fm}^0 = I. 
\end{align}
Using  (\ref{eq:z_recur}) and    $\phi_i^{k+1} \geq I \eta$, for all $i \in \mathcal{V}$ and $k \geq K_1-1$ [due to Lemma \ref{lm:scramb}(b)], we have: 
for $i\in\mathcal{V}$ and $k \geq K_1-1$,
\begin{align}
&\abs{\frac{z_i^{k+1}}{\phi_i^{k+1}} - \frac{\bd{1}^{\top}{\fz}^{k+1}}{I}} \leq  \frac{1}{I \eta} \abs{z_i^{k+1} - \frac{\phi_i^{k+1}}{I}(\bd{1}^{\top}{\fz}^{k+1})}  \nonumber \\
& \leq \frac{1}{I \eta} \abs{z_i^{k+1}  - \xi_i^{k}\bd{1}^{\top}{\fz}^{k+1}}    +  \frac{1}{I \eta}\abs{\left( \xi_i^{k}- \frac{\phi_i^{k+1}}{I} \right) \bd{1}^{\top}{\fz}^{k+1}}  \nonumber   \\
& \leq \frac{1}{I \eta} \abs{z_i^{k+1}  - \xi_i^{k}\bd{1}^{\top}{\fz}^{k+1}} \nonumber\\
& \quad + \frac{1}{I \eta} \abs{\xi_i^{k} - \frac{{\M}_{i,:}^{k:0} {\fm}^0}{I}}\cdot \abs{\bd{1}^{\top} {\fz}^{0}+\sum_{l=0}^k \bd{1}^{\top}  {\gerror}^{l}}  \nonumber\\
& \overset{\eqref{eq:Lemma_exp_decay}}{\leq} \frac{1}{I \eta} \abs{z_i^{k+1}  - \xi_i^{k}\bd{1}^{\top}{\fz}^{k+1}}  
  +  \frac{C \rho^k} {\sqrt{I} \eta }  \left( \norm{{\mathbf{z}}^{0}}+\sum_{l=0}^k \abs{\epsilon^{l}} \right)    \label{eq:track_1}
\end{align}
The next lemma provides a bound of  $\abs{z_i^{k+1}  - \xi_i^{k}\bd{1}^{\top}{\fz}^{k+1}}.$
\begin{lemma}\label{lm:track}
 Let $\{ \fz^k \}_{k=0}^{\infty}$ be the sequence generated by the  perturbed system \eqref{eq:ordy_1}, under Assumption~\ref{ass:delays},     $\bd{A} = (a_{ij})_{i,j=1}^I$ satisfying Assumption~\ref{assumption_weights} (i), (iii), and given     $\{\epsilon^{k}\}_{k\in \mathbb{N}_0}$.
 For any $i\in \mathcal{V}$ and $k\geq 0$, there holds\vspace{-0.2cm}
\begin{align}\label{eq:track_2} \abs{z_i^{k+1}  - \xi_i^{k}\bd{1}^{\top}{\fz}^{k+1}}    \leq  C_0 \left(\rho^{k}\norm{\mathbf{z}^0}+\sum_{l=0}^k\rho^{k-l}\abs{\epsilon^{l}} \right),\end{align}
with  $\{{\bm{\xi}}^k\}_{n\in \mathbb{N}_0}$  defined in Lemma~\ref{lm:rate1} and   $C_0 \triangleq C \sqrt{2S}/\rho.$
\end{lemma}
\begin{proof}\vspace{-0.2cm}
\begin{equation*}
\begin{aligned}
&\abs{z_i^{k+1}  - \xi_i^{k}\bd{1}^{\top}{\fz}^{k+1}}  \overset{\eqref{eq:ordy_1}}{=}  \Bigg| \left({\M}_{i,:} ^{k:0} {\fz}^{0}+\sum_{l=1}^k {\M}_{i,:}^{k:l} {\gerror}^{l-1}+{\lgerror}_i^k \right)  \\ 
& \quad - \xi_i^{k}  \left( \bd{1}^{\top} {\fz}^{0}+\sum_{l=0}^k \bd{1}^{\top}  {\gerror}^{l} \right) \Bigg|  \leq \abs{{\lgerror}_i^k } + \abs{\bd{1}^{\top}  {\gerror}^{k} } \\
& \quad + \norm{{\M}_{i,:} ^{k:0} - \xi_i^{k}  \bd{1}^{\top} } \norm{{\fz}^{0}}  + \sum_{l=1}^k \norm{{\M}_{i,:}^{k:l} -  \xi_i^{k} \bd{1}^{\top} } \norm{{\gerror}^{l-1}}    \\
& \overset{\eqref{eq:Lemma_exp_decay}}{\leq} 
    \frac{ \sqrt{S}}{ \rho} C \left( \rho^k\norm{{\fz}^{0}}+\sum_{l=0}^k \rho^{k-l} \norm{ \Sd^l } \abs{\epsilon^{l}} \right) \\
 &  \overset{(a)}{\leq}   C_0  \left(\rho^{k}\norm{\mathbf{z}^0}+\sum_{l=0}^k\rho^{k-l}\abs{\epsilon^{l}} \right),
\end{aligned}
\end{equation*}
where  in (a)  we used   $\norm{\Sd^l } \leq \sqrt{\norm{\Sd^l }_1 \norm{\Sd^l }_{\infty}} \leq \sqrt{2}.$
\end{proof}
Combing Eq.~\eqref{eq:track_1} and \eqref{eq:track_2} leads to
\begin{equation*}
\begin{aligned}
\abs{\frac{z_i^{k+1}}{\phi_i^{k+1}} - \frac{\bd{1}^{\top}{\fz}^{k+1}}{I}} \leq C_1 \left(\rho^{k}\norm{\mathbf{z}^0}+\sum_{l=0}^k\rho^{k-l}\abs{\epsilon^{l}} \right),
\end{aligned}
\end{equation*}
where we defined $C_1 \triangleq C_0\cdot 2/(I\,\eta).$   

Recalling the definition of $\mathfrak{m}^{k}_{z} \triangleq  \sum_{i=1}^I {z}_i^{k} + \sum_{(j,i)\in \mathcal{E}} (\rho_{ij}^{k} - \tilde{\rho}_{ij}^{k})$, to complete the proof, it remains to show  that \vspace{-0.2cm}  \begin{equation}\label{eq:mean-equiv}
   	\mathfrak{m}^{k}_z \overset{(I)}{=}   \sum_{i=1}^I {{z}_i^0} + \sum_{t=0}^{k-1} \epsilon^t \overset{(II)}{=}  \bd{1}^{\top} \fz^k.
   \end{equation}
We prove next the equalities  (I) and (II)   separately.

 \textit{Proof of (I)}: Since $\mathfrak{m}^{0}_{z} = \sum_{i=1}^I {{z}_i^0}$,   it suffices to show that $\mathfrak{m}^{k+1}_{z} = \mathfrak{m}^{k}_{z} + \epsilon^k$ for all $k\in \mathbb{N}_0$.   
Since agent $i^k$ triggers $k \to k+1,$ we only need to show that
\begin{align*}
& z_{i^k}^{k+1} + \sum_{j\in  \mathcal{N}_{i^k}^{\text{in}}} ({\rho}_{{i^k}j}^{k+1} - \tilde{{\rho}}_{{i^k}j}^{k+1}) + \sum_{j\in  \mathcal{N}_{{i^k}}^{\text{out}}} (\rho_{j {i^k}}^{k+1} - \tilde{\rho}_{j {i^k}}^{k+1}) \nonumber\\
 = & z_{{i^k}}^{k} + \sum_{j\in  \mathcal{N}_{{i^k}}^{\text{in}}} (\rho_{{i^k}j}^{k} - \tilde{\rho}_{{i^k} j}^{k}) + \sum_{j\in  \mathcal{N}_{i^k}^{\text{out}}} (\rho_{j {i^k} }^{k} - \tilde{\rho}_{j {i^k}}^{k}) + \epsilon^k. 
\end{align*}
We have
\begin{align*}
& z_{i^k}^{k+ 1} + \sum_{j\in  \mathcal{N}_{i^k}^{\text{in}}} (\rho_{i^kj}^{k+1} - \tilde{\rho}_{i^k j}^{k+1}) + \sum_{j\in  \mathcal{N}_{i^k}^{\text{out}}} (\rho_{j i^k }^{k+1} - \tilde{\rho}_{j i^k}^{k+1}) \nonumber\\
\overset{(a)}{=}  &  a_{i^ki^k} z_{i^k}^{k+\frac{1}{2}} + \sum_{j\in  \mathcal{N}_{i^k}^{\text{in}}} (\rho_{i^kj}^{k} -  \rho_{i^kj}^{{\tau_{i^k j}^k}} )  \\ & +\sum_{j\in  \mathcal{N}_{i^k}^{\text{out}}} (\rho_{j i^k}^k  + a_{j i^k} z_{i^k}^{k+\frac{1}{2}} - \tilde{\rho}_{j i^k}^{k})  \nonumber\\
\overset{(b)}{=}  & z_{i^k}^{k+\frac{1}{2}} + \sum_{j\in  \mathcal{N}_{i^k}^{\text{in}}} (\rho_{i^kj}^{k} -  \rho_{i^kj}^{ {\tau_{i^k j}^k}} ) + \sum_{j\in  \mathcal{N}_{i^k}^{\text{out}}} (\rho_{j i^k}^k  - \tilde{\rho}_{j i^k}^{k}) \nonumber\\
 \overset{(c)}{=}  & z_{i^k}^{k} + \sum_{j \in \mathcal{N}_{i^k}^{ {\text{in}}}} (\rho_{i^kj}^{ {\tau_{i^k j}^k}} - \tilde{\rho}_{i^kj}^{k})   + \epsilon^k + \sum_{j\in  \mathcal{N}_{i^k}^{\text{in}}} (\rho_{i^kj}^{k} -  \rho_{i^kj}^{ {\tau_{i^k j}^k}} ) \\ & + \sum_{j\in  \mathcal{N}_{i^k}^{\text{out}}} (\rho_{j i^k}^k  - \tilde{\rho}_{j i^k}^{k}) \nonumber\\
  =  &  z_{i^k}^{k} + \sum_{j\in  \mathcal{N}_{i^k}^{\text{in}}} (\rho_{i^k j}^{k} - \tilde{\rho}_{i^k j}^{k}) + \sum_{j\in  \mathcal{N}_{i^k}^{\text{out}}} (\rho_{j i^k }^{k} - \tilde{\rho}_{j i^k}^{k}) + \epsilon^k,
\end{align*}
where  in (a) we used:  the definition of the push step, $\rho_{i^kj}^{k+1} = \rho_{i^kj}^{k}$ for all $j\in \mathcal{N}_{i^k}^{\text{in}}$, and $\tilde{\rho}_{j i^k}^{k+1} = \tilde{\rho}_{j i^k}^{k}$ for all $j\in \mathcal{N}_{i^k}^{\text{out}}$;    (b) follows from   $a_{i^ki^k} + \sum_{j\in \mathcal{N}_{i^k}^{\text{out}}} a_{ji^k}=1$;  and in (c), we used the sum-step.

\textit{Proof of (II):} Using  ~\eqref{eq:z_recur}, yields 
 $\bd{1}^{\top}\fz^{k+1} =  \bd{1}^{\top} {\fz}^{0}+\sum_{l=0}^k \bd{1}^{\top}  {\gerror}^{l}  =  \bd{1}^{\top} \fz^k + \bd{1}^{\top} \perty^k = \sum_{i=1}^I {{z}_i^0} + \sum_{t=0}^{k} \epsilon^t$.\qed

\section{ASY-SONATA--Proof of Theorem~\ref{thm:linear_const}}\label{sec:theory}  We organize the proof in the following steps: \textbf{Step 1:} We  introduce  and study convergence of an auxiliary perturbed consensus scheme,  which serves as a unified model for the descent and consensus updates in  ASY-SONATA--the main result is  summarized in Proposition \ref{lm:PAC}; \textbf{Step 2:} We introduce the consensus and gradient  tracking errors along with a suitably defined optimization error; and we derive bounds connecting these quantities, building on results in Step 1 and convergence of \PertAvgname/--see Proposition \ref{lm:error_bounds}.  The goal is to prove that the aforementioned errors   vanish at a linear rate.
To do so, \textbf{Step 3} introduces a general form of the small gain theorem--Theorem \ref{thm:SGT}--along with  some technical results, which allows us to establish the desired linear convergence through  the boundedness of the solution of an associated  linear system of inequalities.  \textbf{Step 4} builds such a linear system for the error quantities introduced in Step 2  and proves the boundedness of its solution, 
proving thus Theorem \ref{thm:linear_const}.  
The rate expression \eqref{eq:rate-expression} is derived in Appendix \ref{sec:linear_2}.  
Through the proof we assume  $n=1$ (scalar variables); and   define $C_L \triangleq \max_{i=1,\ldots,I} L_i$ and $ L \triangleq  \sum_{i=1}^I L_i.$\vspace{-0.3cm}

\subsection*{\bf Step 1: A perturbed asynchronous    {consensus} scheme}\label{sec:PAC}
We introduce a unified model to study the dynamics of the consensus and optimization errors in ASY-SONATA, which consists in pulling out the tracking update (Step 5) and treat the z-variables--the term $-\gamma^k z_{i^k}^k$ in (\ref{local-descent})--as   an exogenous perturbation $\delta^k$.  More specifically, consider the following scheme (with a slight abuse of notation, we use the same symbols as in ASY-SONATA): \vspace{-0.1cm}\begin{subequations}\label{eq:PAC}
\begin{align} 
& v_{i^k}^{k+1}=x_{i^k}^k + \delta^k, \\
& x_{i^k}^{k+1} = w_{i^k i^k} v_{i^k}^{k+1} +  \sum_{j\in \mathcal{N}_{i^k}^{\text{in}}} w_{{i^k}j} v_{j}^{k - d_{j}^k},\\
& v_{j}^{k+1} = v_{j}^{k}, \, x_{j}^{k+1} = x_{j}^{k}, \quad \forall j \in \mathcal{V} \setminus \{i^k\},
\end{align}\end{subequations}
with given $x_i^0 \in \mathbb{R}$,   $v_i^t = 0$, $t = -D,-D+1,\ldots,0$, for all $i\in \tilde{\mathcal{V}}$. We make the blanket assumption that agents' activations and delays satisfy   Assumption~\ref{ass:delays}. 

Let us rewrite (\ref{eq:PAC}) in a vector-matrix form. Define $\bd{x}^k \triangleq [{x}_1^k, \cdots, {x}_I^k]^{\top}$ and $\bd{v}^k \triangleq [{v}_1^k, \cdots, {v}_I^k]^{\top}$. Construct the   $(D+2)I$ dimensional concatenated vectors    
\begin{equation}\label{eq:h_vector}
 \hspace{-0.2cm}\augX^{k} \triangleq [{\bd{x}^k}^{\top}, {\bd{v}^k}^{\top}, {\bd{v}^{k-1}}^{\top}, \cdots, {\bd{v}^{k-D}}^{\top}]^{\top}, \quad 
 \augY^k \triangleq \delta^k \, \mathbf{e}_{i^k};
\end{equation}
and the 
  augmented matrix $\W^k$, defined as
\[
\eW_{rm}^k  \triangleq 
\left\{
\begin{aligned}
& w_{i^k i^k}  ,  && \text{if }  r=m = i^k ; \\
& w_{i^k j},        && \text{if }  r = i^k,\, m= j+ (d_{j}^{k}+1) I  ; \\
& 1 ,  && \text{if }  r = m \in \{1,2,\ldots,2I\} \setminus \{i^k, i^k+I\};  \\
& 1 ,  && \text{if }  r \in  \{2I+1,2I+2,\ldots, (D+2)I\} \\
&      &&  \cup \{i^k+I\} \text{ and } m = r-I;\\
& 0,   && \text{otherwise}.
\end{aligned}
\right.
\]

System \eqref{eq:PAC} can be rewritten in compact form as 
\begin{align}\label{eq:chi}
\bd{\augX}^{k+1} = \W^k(\bd{\augX}^{k} + \augY^k),
\end{align}
 The following lemma captures the asymptotic behavior of   $\W^k$. 

\begin{lemma}\label{lm:Wscramb}
Let $\{\W^k \}_{k\in \mathbb{N}_0}$ be  the sequence of matrices in \eqref{eq:chi}, generated by \eqref{eq:PAC},   under Assumption~\ref{ass:delays} and with   $\mathbf{W}$ satisfying Assumption~\ref{assumption_weights} (i), (ii).
 The following hold: for all  $k\in \mathbb{N}_0$,  a)  $\W^k$ is row stochastic;  b) there exists a sequence of stochastic vectors $\{\bm{\psi}^k\}_{k\in \mathbb{N}_0}$ such that  
\begin{equation}\label{eq:augmented_W}
	\hspace{-0.2cm}\norm{\W^{k:t} -\bd{1}{\bm{\psi}^t}^{\top}} \leq C_2 \rho^{k-t},\quad C_2 \triangleq \frac{2\sqrt{(D+2)I}  (1+ \lbm^{-K_1})}{1- \lbm^{-K_1}}
\end{equation} 
Furthermore, $\psi_i^k\geq \lbpsi = \lbm^{K_1}$, for all $k\geq 0$ and  $i\in \mathcal{V}$.
\end{lemma}
\begin{proof}
The proof follows similar techniques as in \cite{nedic2010convergence,lin2013constrained}, and can be found in Appendix~\ref{pf:wscramb}.
\end{proof}

We define now a proper consensus error for (\ref{eq:chi}). 
Writing $\bd{\augX}^{k}$ in \eqref{eq:chi} recursively, yields\vspace{-0.2cm}
\begin{equation}\label{chi}
\bd{\augX}^{k+1} = \W^{k:0}\bd{\augX}^0+\sum_{l=0}^k  \W^{k:l}\augY^l.\vspace{-0.1cm}
\end{equation}
Using   Lemma \ref{lm:Wscramb}, for any fixed $N\in \mathbb{N}_0$, we have  $$\lim_{k\to\infty} (\W^{k:0}\bd{\augX}^0+\sum_{l=0}^N  \W^{k:l}\augY^l)=\bd{1}{\bm{\psi}^0}^{\top}\bd{\augX}^0+\sum_{l=0}^N  \bd{1}{\bm{\psi}^l}^{\top}\augY^l.$$
Define \vspace{-0.2cm}
\begin{equation}\label{eq:dfz}
\CV^{0} \triangleq {\bm{\psi}^0}^{\top}\bd{\augX}^0, \quad \CV^{k+1} \triangleq {\bm{\psi}^0}^{\top}\bd{\augX}^0+\sum_{l=0}^k  {\bm{\psi}^l}^{\top}\augY^l, \,\, k\in \mathbb{N}_0.\vspace{-0.2cm}
\end{equation}
Applying \eqref{eq:dfz} inductively,  it  is easy to check that  
\begin{equation}\label{dyz}
\CV^{k+1} = \CV^{k} + {\bm{\psi}^k}^{\top}\augY^k = \CV^{k} +  \psi_{i^k}^k \delta^k.
\end{equation}
We are now ready to state the main result of this  subsection, which is a bound of the consensus disagreement $\|\augX^{k+1}-\bd{1}\CV^{k+1}\|$  in terms of the magnitude of the perturbation. 

\begin{proposition}\label{lm:PAC}
In the above setting, the consensus error $\|\augX^{k+1}-\bd{1}\CV^{k+1}\|$ satisfies: for all $k\in \mathbb{N}_0$, \vspace{-0.2cm}
\begin{equation*}
 \norm{\augX^{k+1}-\bd{1}\CV^{k+1}} 
\leq  C_2 \rho^k \norm{\bd{\augX}^0-\bd{1}\CV^{0}} + C_2 \sum_{l=0}^k \rho^{k-l} \abs{\delta^l}.\vspace{-0.2cm}
\end{equation*}
\end{proposition}

\begin{proof}
The proof follows readily from   \eqref{chi},  \eqref{eq:dfz}, and Lemma \ref{lm:Wscramb}; we omit further details.
\vspace{-0.2cm}\end{proof}

\subsection*{\bf Step 2: Consensus,     tracking, and optimization errors}\label{sec:pre}

 \noindent \textbf{1) Consensus disagreement}: As anticipated,   the updates of ASY-SONATA are also described by   \eqref{eq:PAC}, if one sets  therein $\delta^k = -\gamma^k {z}_{i^k}^k$ (with  ${z}_{i^k}^k$ satisfying  Step 5 of ASY-SONATA). Let  $\mathbf{h}^k$ and $\CV^{k}$ be defined as in    \eqref{eq:h_vector} and (\ref{eq:dfz}), respectively,  with $\delta^k = -\gamma^k {z}_{i^k}^k$. 
The  consensus error  at iteration $k$ is defined as  \vspace{-0.2cm}
\begin{equation}\label{eq:consensus-error}
	\CE^k \triangleq \norm{\augX^{k}-\bd{1}\CV^{k}}.\vspace{-0.1cm}
\end{equation}  
 
 \noindent \textbf{2) Gradient tracking error}:  The gradient tracking step in ASY-SONATA is an instance of  \PertAvgname/, with  $\epsilon^k =  \nabla f_{i^k}(\bd{x}_{i^k}^{k+1}) - \nabla f_{i^k}(\bd{x}_{i^k}^{k})$. By  Proposition \ref{prop_equivalence},  \PertAvgname/ is equivalent to \eqref{eq:augmented_consensus}. In view of  Lemma~\ref{lm:track} and the following property 
   $\mathbf{1}^\top \fz^k  = \sum_{i=1}^I \nabla f_i (x_i^0) + \sum_{t=0}^{k-1} \left( \nabla f_{i^t}(x_{i^t}^{t+1}) - \nabla f_{i^t}(x_{i^t}^{t}) \right) = \sum_{i=1}^I \nabla f_i (x_i^k)$ where the first equality follows from \eqref{eq:mean-equiv} and $\epsilon^k =  \nabla f_{i^k}(\bd{x}_{i^k}^{k+1}) - \nabla f_{i^k}(\bd{x}_{i^k}^{k})$ while in the second equality we used $x_j^{t+1} = x_j^t $, for $j\neq i^t$,  the tracking error at iteration $k$ along with the magnitude of the tracking variables are defined as  
  \vspace{-0.2cm}\begin{equation}\label{eq:tracking-error}
\hspace{-0.2cm}\TE^k \triangleq \abs{{z}_{i^k}^{k}- \xi_{i^k}^{k-1}\,\bar{g}^k}, \,\, \NT^k \triangleq \abs{{z}_{i^k}^{k}},\,\,\bar{g}^k \triangleq \sum_{i=1}^I \nabla f_i (x_i^k),\vspace{-0.2cm} \end{equation} with $\xi_i^{-1} \triangleq  \eta$,   $i\in \mathcal{V}$. Let   $\bd{g}\spe{k} \triangleq [\nabla f_1(x_1^k),\ldots, \nabla f_I(x_I^k)]^{\top}.$
 
 \noindent \textbf{3) Optimization error:} Let  $x^\star$ be the unique minimizer of $F$. Given the definition of consensus disagreement in \eqref{eq:consensus-error}, 
  we define the optimization error at iteration $k$ as  \begin{equation}\label{eq:optimization-error}
		\OG^k \triangleq  \abs{\CV^k - x^\star}.
	\end{equation} 
 Note that this is a natural choice as, if consensual, all agents' local variables will converge to a limit point of $\{\CV^k\}_{k\in \mathbb{N}_0}$. 
 
\noindent \textbf{4) Connection among $\CE^k$, $\TE^k$, $\NT^k$, and $\OG^k$:} The following proposition establishes bounds on the above quantities.

 \begin{proposition}\label{lm:error_bounds} Let $\{\bx^k,\bv^k, \bz^k\}_{k\in \mathbb{N}_0}$ be the sequence generated by ASY-SONATA, in the setting of  {Theorem~\ref{thm:linear_const}}, but possibly with a time-varying step-size $\{\gamma^k\}_{k\in \mathbb{N}_0}$. The error quantities    $\CE^k$, $\TE^k$, $\NT^k$, and $\OG^k$ satisfy: for all $k\in \mathbb{N}_0$, \vspace{-0.2cm}\begin{subequations} \label{eq:error_bounds}
 \begin{alignat}{2}
 	& \CE^{k+1} \leq  && C_2 \rho^k \CE^{0} + C_2 \sum_{l=0}^k \rho^{k-l} \gamma^l \NT^l.\label{lm:inter_eq1} 	\\ 
	&\TE^{k+1}  \leq  && 3 C_0 C_L \sum_{l=0}^k \rho^{k-l}  \left(  \CE^l  + \gamma^l  \NT^l \right)    + C_0 \rho^k \norm{\bd{g}\spe{0}};\label{lm:inter_eq2}\\
  & \NT^{k} \,\, \leq  && \TE^{k}   +  C_L \sqrt{I} \CE^{k} + L \,\OG^{k}\label{eq:sys_ine_trac} 
  \end{alignat}
  Further assume $\gamma^k \leq 1/L,\,  k \in \mathbb{N}_0;$ then\vspace{-0.2cm}
   \begin{alignat}{2}
 &   \OG^{k+1}  {\leq} && \sum_{l=0}^k  \bigg( \prod_{t = l+1}^k \left(1- \tau \lbpsi^2 \gamma^t \right) \bigg)  \big(C_L\sqrt{I}  \CE^l +\TE^l \big) \gamma^l     \nonumber \\
  & &&   +  \prod_{t = 0}^k \left(1- \tau \lbpsi^2 \gamma^t  \right) \OG^0,  \label{eq:sys_ine_cont} 
\end{alignat}
\end{subequations}
where $\eta\in (0,1)$ is defined in Lemma \ref{lm:rate1} and $\tau$ is the strongly convexity constant of $F$.  
\end{proposition}
\begin{proof} Eq. \eqref{lm:inter_eq1} follows readily from 
 Proposition \ref{lm:PAC}.
 
 We prove now  \eqref{lm:inter_eq2}. Recall $\mathbf{1}^\top \fz^k  =    \bar{g}^{k}$.  Using Lemma \ref{lm:track} with $\epsilon^k =  \nabla f_{i^k}(\bd{x}_{i^k}^{k+1}) - \nabla f_{i^k}(\bd{x}_{i^k}^{k})$, we obtain:  for all $i \in \mathcal{V},$
\begin{align*}
&\abs{{z}_{i}^{k+1}- \xi_{i}^{k}\,\bar{g}^{k+1} } \\
& \leq   C_0  \left(\rho^k \norm{\bd{g}\spe{0}} + \sum_{l=0}^k \rho^{k-l} \abs{\nabla f_{i^l}^{l+1} - \nabla f_{i^l}^{l}}\right) \\
& \leq  C_0 \, \rho^k \norm{\bd{g}\spe{0}}  + C_0 C_L \sum_{l=0}^k \rho^{k-l} \abs{ {x}_{i^l}\spe{l+1}-{x}_{i^l}\spe{l}} \\
& \leq  C_0 \, \rho^k \norm{\bd{g}\spe{0}}  + C_0 C_L \sum_{l=0}^k \rho^{k-l} \norm{ \augX^{l+1}-\augX^{l}} \\ 
& = C_0 \rho^k \norm{\bd{g}\spe{0}}  + C_0 C_L \sum_{l=0}^k \rho^{k-l} \norm{ \W^l \left( \bd{\augX}^{l}+ \augY^l \right)  - \augX^{l}} \\  
&\overset{(a)}{=}  C_0 \, \rho^k \norm{\bd{g}\spe{0}}\\
&  +  C_0 C_L \sum_{l=0}^k \rho^{k-l} \norm{ \big(\W^l-\bd{I}\big)  \big(\bd{\augX}^{l}-\bd{1}\CV^l \big) -  \gamma^l z_{i^l}^l \W^l  \mathbf{e}_{i^l}} \\  
&\leq C_0 \rho^k \norm{\bd{g}\spe{0}}  + C_0 C_L \sum_{l=0}^k \rho^{k-l} \Bigg(  \|\W^l\| \gamma^l \NT^l  \\
& + \bigg( \| \W^l\| + \|\bd{I}\| \bigg) \CE^l \Bigg)  \\  
&\overset{(b)}{\leq} C_0 \rho^k \norm{\bd{g}\spe{0}}  + 3 C_0 C_L \sum_{l=0}^k \rho^{k-l} \left(  \CE^l + \gamma^l  \NT^l \right),  
\end{align*}
where in (a) we used \eqref{eq:chi} and the row stochasticity of $\W^k$ [Lemma \ref{lm:Wscramb}(a)]; and   (b) follows from  $\|\W^l\|  \leq \sqrt{\| \W^l\|_1 \| \W^l \|_{\infty} } \leq \sqrt{3}.$ This proves    \eqref{lm:inter_eq2}.

Eq. \eqref{eq:sys_ine_trac} follows readily from  \begin{align*}  \NT^k &=  \abs{z_{i^k}^k} \leq  \abs{{z}_{i^k}^{k}- \xi_{i^k}^{k-1}\,\avg{g}^k} + \xi_{i^k}^{k-1} \abs{\avg{g}^k- \nabla F(\CV^k) } \\
&\qquad\qquad   + \xi_{i^k}^{k-1}  \abs{\nabla F(\CV^k) - \nabla F(x^\star)}.
\end{align*}
  Finally, we prove (\ref{eq:sys_ine_cont}). 
 Using \eqref{eq:sys_ine_trac}  and  
$\CV^{k+1} =  \CV^{k} - \gamma \psi_{i^k}^k z_{i^k}^k$ [cf. (\ref{dyz}) and recall $\delta^k=-\gamma z_{i^k}^k$],
we can write
\begin{equation*}
\begin{aligned}
& \OG^{k+1} =    \abs{\CV^{k} - \gamma^k \psi_{i^k}^k z_{i^k}^k- x^\star} \\
& \leq  \gamma^k \psi_{i^k}^k \xi_{i^k}^{k-1} \abs{  \nabla F(\CV^k) -  \avg{g}^k} + \gamma^k \psi_{i^k}^k \abs{\xi_{i^k}^{k-1}  \, \avg{g}^k - z_{i^k}^k}     \\
& +   \abs{\CV^{k} -  \gamma^k \psi_{i^k}^k \xi_{i^k}^{k-1} \nabla F(\CV^k)- x^\star} \\
&  \overset{(a)}{\leq}  \left(1- \tau \lbpsi^2 \gamma^k \right)  \OG^k + C_L\sqrt{I} \gamma^k \norm{\augX^k - \bd{1}\CV^k}   +  \gamma^k \TE^k \\
& \overset{(b)}{\leq}   \sum_{l=0}^k  \bigg( \prod_{t = l+1}^k \left(1- \tau \lbpsi^2 \gamma^t \right) \bigg)  \big(C_L\sqrt{I}  \CE^l +\TE^l \big) \gamma^l     \nonumber \\
 &   +  \prod_{t = 0}^k \left(1- \tau \lbpsi^2 \gamma^t  \right) \OG^0
\end{aligned}
\end{equation*}
where in $(a)$ we used $\lbpsi^2 \leq \psi_{i^k}^k \xi_{i^k}^{k-1}   < 1$ (cf. Lemma \ref{lm:rate1}) and $\abs{x- \gamma \nabla F(x) -   x^\star} \leq (1-\tau \gamma) \abs{x - x^\star},$ which holds for $\gamma \leq  {1}/{L}$;  $(b)$ follows readily by applying  the above inequality telescopically. \vspace{-0.3cm}\end{proof} 
 \subsection*{\bf Step 3: The generalized small gain theorem}\label{sec:sgt}
  The last step of our proof is to show that the error quantities   $\CE^k$, $\TE^k$, $\NT^k$, and $\OG^k$  vanish linearly. This is not a straightforward task, as these quantities are interconnected through the inequalities \eqref{eq:error_bounds}. This subsection provides tools to address this issue. The key result is a generalization of the small gain theorem (cf. Theorem \ref{thm:SGT}), first used in \cite{Nedich-geometric}.

\begin{definition}[\!\cite{Nedich-geometric}]\label{df:sgt}
Given the sequence $\{u^k \}_{k=0}^{\infty}$,   a constant  $\lambda \in (0,1)$, and   $N \in \mathbb{N}$, let us define
$$\abs{u}^{\lambda,N} = \max_{k = 0,\ldots, N} \frac{\abs{u^k}}{\lambda^k}, \quad \abs{u}^{\lambda} = \sup_{k \in \mathbb{N}_0} \frac{\abs{u^k}}{\lambda^k}.$$
If $\abs{u}^{\lambda} $ is upper bounded, then $u^k  = \mathcal{O} (\lambda^k)$, for all $k\in \mathbb{N}_0.$
\end{definition}

The following lemma shows how one can interpret the inequalities in \eqref{eq:error_bounds}  using the notions introduced in Definition~\ref{df:sgt}. 

\begin{lemma}\label{conv_lm}
Let $\{u^k\}_{k=0}^{\infty}, \{v^k_i\}_{k=0}^{\infty}, i = 1,\ldots, m$, be nonnegative sequences; let   $\lambda_0, \lambda_1, \ldots, \lambda_m \in (0,1)$; and let $R_0,R_1, \ldots, R_m \in \mathbb{R}_+$ such that \vspace{-0.2cm}$$u^{k+1} \leq R_0 (\lambda_0) ^k + \sum_{i = 1}^m R_i \sum_{l=0}^{k} (\lambda_i)^{k-l} v_i^l,\quad \forall k\in \mathbb{N}_0.\vspace{-0.1cm}$$
 Then, there holds  \vspace{-0.1cm}
$$\abs{u}^{\lambda,N}  \leq u^0 +  \frac{R_0 }{\lambda} + \sum_{i = 1}^m \frac{R_i}{\lambda-\lambda_i}  \abs{v_i}^{\lambda,N},\vspace{-0.1cm}$$ for any $\lambda \in (\displaystyle\max_{i = 0, 1, \ldots, m} \lambda_i, 1)$ and   $N\in \mathbb{N}$.
\end{lemma}

\begin{proof} See Appendix~\ref{pf:conv}.
\end{proof}

\begin{lemma}\label{norm_lm}
Let   $\{u^k\}_{k=0}^{\infty}$ and $\{v^k\}_{k=0}^{\infty}$ be two nonnegative sequences. The following hold
\begin{enumerate}
\item   $u^k \leq v^k$,  for all $k \in \mathbb{N}_0 \Longrightarrow$     $\abs{u}^{\lambda,N} \leq \abs{v}^{\lambda,N},$ for any  $\lambda \in (0, 1)$ and  $N\in \mathbb{N}$;
\item  \vspace{-0.2cm}$$\abs{\beta_1 u + \beta_2 v}^{\lambda,N} \leq \abs{\beta_1}\abs{u}^{\lambda,N} + \abs{\beta_2}\abs{v}^{\lambda,N},$$ for any $\beta_1,\beta_2 \in \mathbb{R}$,  $\lambda \in (0, 1)$, and  positive integer $N$. 
\end{enumerate}
\end{lemma}
\begin{figure*}\vspace{-0.8cm}
\begin{equation}\label{eq:matrix}
 \left[ {\begin{array}{c}
   \abs{\NT}^{\lambda, N} \\
   \abs{\CE}^{\lambda, N} \\
   \abs{\TE}^{\lambda, N} \\
   \abs{\OG}^{\lambda, N} \\
  \end{array} } \right]
\preccurlyeq
\underbrace{
\left[{\begin{array}{cccc}
   0 & b_1  & 1 & L \\
  \frac{C_2 \gamma}{\lambda-\rho} & 0 & 0 & 0  \\
   \frac{b_2 \gamma}{\lambda-\rho} & \frac{b_2 }{\lambda-\rho}  & 0 & 0  \\
   0 & \frac{{b_2} \gamma}{\lambda - \mathcal{L}(\gamma)} & \frac{\gamma}{\lambda - \mathcal{L}(\gamma)} & 0\\
  \end{array} }
\right]}_{\triangleq \bd{K}}
 \left[ {\begin{array}{c}
   \abs{\NT}^{\lambda, N} \\
   \abs{\CE}^{\lambda, N} \\
   \abs{\TE}^{\lambda, N} \\
   \abs{\OG}^{\lambda, N} \\
  \end{array} } \right]
+
\left[ {\begin{array}{c}
   0 \\
    \Big(1+ \frac{C_2}{\lambda}\Big)\CE^0  \\
   \frac{C_0 \norm{\bd{g}^{0}}}{\lambda}+\TE^0  \\
    \frac{1+ {\lambda}}{{\lambda}} \OG^0  \\
  \end{array} } \right],\quad b_1 \triangleq C_L\sqrt{I}, \quad b_2 \triangleq 3 C_0 C_L.  \end{equation}
 \begin{align*}
\end{align*}\vspace{-2cm}
\end{figure*}

The major result of this section is the   generalized small gain theorem, as stated next.
\begin{theorem}\label{thm:SGT}
(Generalized Small Gain Theorem)
 Given   nonnegative sequences $\{u_i^k \}_{k=0}^{\infty}, i = 1,\ldots, m$,   a non-negative matrix $\bd{T} \in \mathbb{R}^{m \times m}$,   $ \bm{\beta} \in \mathbb{R}^{m}$, and   $\lambda \in (0,1)$     such that\vspace{-0.1cm}  \begin{equation}\label{eq:SGT}
\bd{u}^{\lambda, N} \preccurlyeq \bd{T} \bd{u}^{\lambda, N} + \bm{\beta},\quad \forall N\in \mathbb{N},
\end{equation}
where  $\bd{u}^{\lambda, N} \triangleq [ \abs{u_1}^{\lambda,N}, \ldots, \abs{u_m}^{\lambda,N}]^{\top}$. 
If   $\rho(\bd{T})<1$,  then  $\abs{u_i}^{\lambda}$ is bounded,  for all $i=1, \ldots, m$. That is, each $u_i^k$ vanishes at a   R-linear rate $\mathcal{O}(\lambda^k).$\vspace{-0.1cm}
\end{theorem}
 \begin{proof} See Appendix~\ref{pf:sgt}.
\end{proof}

Then following results are instrumental to find a sufficient condition for  $\rho(\bd{T})<1$.

\begin{lemma}\label{lm:polynomial}
Consider a polynomial $p(z) = z^m - a_1 z^{m-1} - a_2 z^{m-2} - \ldots - a_{m-1} z - a_m$, with   $z\in \mathbb{C}$ and $a_i \in \mathbb{R}_+$, $ i = 1,\ldots m.$ Define $z_p \triangleq \max \left \{\abs{z_i} \,\big\vert \,  p(z_i) = 0, \,\ i = 1,\ldots, m  \right \}$. 
Then,   $z_p<1$ if and only if $p(1)>0.$
\end{lemma}
\begin{proof} See Appendix~\ref{pf:polynomial}.
\end{proof}

\subsection*{\bf Step 4: Linear convergence rate (proof of Theorem~\ref{thm:linear_const})} \label{sec:linear}
Our path to prove linear convergence rate passes through Theorem \ref{thm:SGT}: we first cast    the set of inequalities \eqref{eq:error_bounds}  into   a system in the form (\ref{eq:SGT}), and then study the spectral properties of the resulting coefficient matrix. 
 
Given $\gamma<1/L$, define   $\mathcal{L}(\gamma) \triangleq 1- \tau \lbpsi^2 \gamma$; and choose  $\lambda\in \mathbb{R}$  such that \vspace{-0.1cm}
\begin{equation}\label{eq:lambda_interval}
\max \left( \rho, \mathcal{L}(\gamma) \right) < \lambda <1.\vspace{-0.1cm}
\end{equation}
Note that   $\mathcal{L}(\gamma)<1$, as $\gamma<1/L$; hence \eqref{eq:lambda_interval} is nonempty.  

Applying Lemma~\ref{conv_lm} and Lemma~\ref{norm_lm} to the set of inequalities \eqref{eq:error_bounds}  {with $\gamma^k \equiv \gamma$}, we obtain  the system ~\eqref{eq:matrix} at the top of the page. By    Theorem~\ref{thm:SGT}, to prove the desired linear convergence rate,     it is sufficient to show that $\rho(\mathbf{K})<1$.  The characteristic polynomial  $p_{\mathbf{K}}(t)$ of  $\bd{T}$ satisfies the conditions of   Lemma~\ref{lm:polynomial}; hence $\rho(\mathbf{K})<1$  if and only if  $p_{\mathbf{K}}(1)>0$, that is,
\begin{equation}\label{eq:rate_func}
\begin{aligned}
  & \left(   \left( 1+ \frac{L \gamma}{\lambda - \mathcal{L}(\gamma)}  \right) \frac{b_2}{\lambda-\rho} + b_1 + \frac{L b_2 \gamma}{\lambda - \mathcal{L}(\gamma)}  \right)  \frac{C_2 \gamma}{\lambda-\rho} \\
   &   +   \left( {1}+ \frac{L \gamma}{\lambda - \mathcal{L}(\gamma)}  \right) \frac{b_2 \gamma}{\lambda-\rho}  \triangleq \mathfrak{B}(\lambda;\gamma)  < 1.
\end{aligned}
\end{equation}
By the continuity of  $\mathfrak{B}(\lambda;\gamma)$ and (\ref{eq:lambda_interval}), 
 $\mathfrak{B}(1;\gamma) <1$ is sufficient to claim the existence of some $\lambda \in \left(\max \left(\rho,\mathcal{L}(\gamma) \right), 1 \right)$ such that $ \mathfrak{B}(\lambda;\gamma) <1$.   Hence, setting  $\mathfrak{B}(1;\gamma) <1$, yields $0< \gamma <\bar{\gamma}_1$, with 
\begin{equation}\label{eq:gamma_1}
 \bar{\gamma}_1\triangleq  \frac{ \tau \lbpsi^2 (1-\rho)^2}{(\tau \lbpsi^2 + L )b_2 (C_2 + 1 -\rho) + (b_1 \tau \lbpsi^2 + L b_2 ) C_2 (1-\rho)}.
 \end{equation}
It is easy to check that $\bar{\gamma}_1 < {1}/{L}$.  Therefore, $0< \gamma <\bar{\gamma}_1$
is   sufficient  for $\CE^k, \TE^k , \NT^k, \OG^k$ to vanish with an R-Linear rate.  The desired result,   $ \abs{x_i^k - x^\star }  = \mathcal{O}(\lambda^k)$,   $i\in \mathcal{V}$, follows readily from $\CE^k = \mathcal{O}(\lambda^k) $ and $\OG^k = \mathcal{O}(\lambda^k).$  The explicit expression of the rate $\lambda$, as in \eqref{eq:rate-expression}, is derived  in  Appendix~\ref{sec:linear_2}. 
\section{ASY-SONATA: Proof of Theorems~\ref{thm:sublinear_const} and  \ref{thm:dimi_sublinear}}\label{sec:sublinear}
Through the section, we use the same notation as in Sec.\ref{sec:theory}.    \vspace{-0.7cm}
\subsection{Preliminaries} 
We begin establishing a connection between the merit function $M_F$  defined in \eqref{eq:merit_function} and the error quantities $\CE^k$, $\TE^k$, and $\NT^k$, defined in \eqref{eq:consensus-error}, \eqref{eq:tracking-error},   and \eqref{eq:optimization-error} respectively.  
 \begin{lemma}\label{lm:merit}
   The merit function $M_F$    satisfies\vspace{-0.1cm}
\begin{equation}\label{lim:square:MF}
M_F(\bx^k) \leq C_3 \,({\CE^k})^2 + 3 \,\eta^{-2}\, \left( ({\TE^k})^2   + ({\NT^k})^2 \right),   
\end{equation}
with   $C_3 \triangleq 3 C_L^2 I + \frac{3L^2}{I} +6C_L L + 4.$
\end{lemma}
\begin{proof}
Define $\bd{J} \triangleq ({1}/{I})\cdot \bd{1} \bd{1}^{\top}$ and $\bar{x}^k\triangleq ({1}/{I})\cdot\bd{1}^{\top}\bx^k$; and recall the definition of $\xi_{i}^{k}$ (cf. Lemma \ref{lm:rate1}) and that  $\CV^{k+1} =  \CV^{k} - \gamma^k \psi_{i^k}^k z_{i^k}^k.$ [cf. \eqref{dyz}].  We have 
\begin{align}
  & M_F(\bx^k)  \leq  \abs{\nabla F(\bar{x}^k) }^2 + \norm{ \bd{x}^k- \bd{1} \bar{x}^k}^2\\
& \leq  \abs{\nabla F(\bar{x}^k) }^2 + 2 \norm{ \bd{x}^k- \bd{1}\CV^k}^2 + 2 \norm{\bd{1}\CV^k - \bd{1} \bar{x}^k}^2 \\
&\leq   \abs{\nabla F(\bar{x}^k) }^2 + 2 \norm{ \bd{x}^k- \bd{1}\CV^k}^2 + 2 \norm{\bd{J} \left(  \bd{1} \CV^k - \bd{x}^k \right)}^2 \nonumber\\
\label{eq:merit_I}
& \leq  \abs{\nabla F(\bar{x}^k) }^2 + 4 \norm{ \bd{x}^k- \bd{1}\CV^k}^2 . 
\end{align}
We bound now $ \abs{\nabla F(\bar{x}^k) }$; we have 
\begin{equation}\label{eq:merit_II} \begin{aligned}
 & \abs{\nabla F(\bar{x}^k)} \leq   
   \abs{\nabla F(\CV^k)} +  L \abs{\bar{x}^k - \CV^k }\\
& \leq  \abs{\nabla F(\CV^k) - \bar{g}^k}   + \abs{\bar{g}^k  - (\xi_{i^k}^{k-1})^{-1}z_{i^k}^k} + (\xi_{i^k}^{k-1})^{-1} \abs{z_{i^k}^k} \\
& \quad  + \frac{L}{\sqrt{I}} \norm{\bd{J} \left( \bd{x}^k - \bd{1} \CV^k \right)}\\
& \leq  
 \left( C_L \sqrt{I} + \frac{L}{\sqrt{I}}  \right) \,\CE^k + \eta^{-1} \,\TE^k   + \eta^{-1}\, \NT^k,
\end{aligned}\end{equation}
where in the last inequality we used $\xi_{i^k}^{k}\geq \eta$ for all $k$ (cf. Lemma \ref{lm:rate1}) and $\|\bd{J}  ( \bd{x}^k - \bd{1} \CV^k )\|\leq \CE^k$.

Eq. (\ref{lim:square:MF}) follows readily from (\ref{eq:merit_I}) and (\ref{eq:merit_II}).  
 \end{proof}

Our ultimate goal is to show that  the RHS of \eqref{lim:square:MF} is summable.   To do so, we need two further results,  Proposition \ref{eq:descent} and Lemma \ref{lm:sqauresum} below.   Proposition \ref{eq:descent}  establishes a connection between   $F(\CV^{k+1})$ and 
$\CE^k$, $\TE^k$, and $\NT^k$. \vspace{-0.1cm}
 
 \begin{proposition}\label{eq:descent} In the above setting, there holds:   $k\in \mathbb{N}_0$,\vspace{-0.2cm}
 \begin{equation}\label{eq:lya}
\begin{aligned}
&  F(\CV^{k+1}) \leq  F(\CV^0) + \frac{1}{2} \left( L + \alpha^{-1} + \beta^{-1}  \right) \sum_{t=0}^k ({\NT^t})^2 ({\gamma^t})^2 \\
& -  \lbpsi  \sum_{t=0}^k ({\NT^t})^2 \gamma^t    +  \frac{\alpha}{2} C_L^2 I \sum_{t=0}^k ({\CE^t})^2+   \frac{\beta}{2}  \eta^{-2} \sum_{t=0}^k ({\TE^t})^2, \\
\end{aligned}
\end{equation} where $\alpha$ and  $\beta$ are two arbitrary  positive constants.  \vspace{-0.1cm}
\end{proposition}
\begin{proof}

By descent lemma, we get \vspace{-0.1cm}
\begin{align*}
 & F(\CV^{k+1}) \leq  \\
& F(\CV^k) + \gamma^k {\psi}_{i^k}^k\inn{\nabla F(\CV^k)}{-{z}_{i^k}^k}    + \frac{L(\gamma^k\psi_{i^k}^k)^2}{2} \abs{z_{i^k}^k}^2 \\%
 &\leq  F(\CV^k) + \frac{L{\gamma^k}^2}{2} \abs{z_{i^k}^k}^2 + \gamma^k \psi_{i^k}^k\inn{(\xi_{i^k}^{k-1})^{-1} z_{i^k}^k }{-z_{i^k}^k}  \\
&  + \gamma^k\psi_{i^k}^k \inn{\nabla F(\CV^k)-\avg{g}^k}{-z_{i^k}^k}   \\
& + \gamma^k\psi_{i^k}^k \inn{\avg{g}^k-(\xi_{i^k}^{k-1})^{-1} z_{i^k}^k }{-z_{i^k}^k} \\%
& \leq  F(\CV^k) + \frac{L{\gamma^k}^2}{2} \abs{z_{i^k}^k}^2 - \gamma^k \lbpsi  \abs{z_{i^k}^k}^2   \\
& + \gamma^k  C_L \sum_{j=1}^I  \abs{\CV^k-x_j^k}\abs{z_{i^k}^k}  + \gamma^k \eta^{-1} \TE^k \abs{z_{i^k}^k} \\& \leq  F(\CV^k) + \frac{L{\gamma^k}^2}{2} \abs{z_{i^k}^k}^2 - \gamma^k \lbpsi  \abs{z_{i^k}^k}^2   \\
 & +  \gamma^k C_L  \sqrt{I} \CE^k \abs{z_{i^k}^k}  + \gamma^k \eta^{-1} \TE^k \abs{z_{i^k}^k}  \\&  \leq  F(\CV^k) + \frac{L{\gamma^k}^2}{2} \abs{z_{i^k}^k}^2 - \gamma^k \lbpsi  \abs{z_{i^k}^k}^2   +  \frac{\alpha}{2} C_L^2 I  ({\CE^k})^2\\  
& +  \frac{\alpha^{-1}}{2}  \abs{z_{i^k}^k}^2 {\gamma^k}^2  + \frac{\beta}{2}  \eta^{-2} ({\TE^k})^2 + \frac{\beta^{-1} }{2} \abs{z_{i^k}^k}^2 ({\gamma^k})^2 \\
&  \leq   F(\CV^k) + \frac{1}{2} \left( L + \alpha^{-1} + \beta^{-1}  \right) ({\NT^k })^2 ({\gamma^k})^2  \\
& -  \lbpsi  ({\NT^k })^2 \gamma^k +  \frac{\alpha}{2} C_L^2 I  ({\CE^k})^2  +   \frac{\beta}{2}  \eta^{-2} ({ \TE^k})^2.  
\end{align*}\nonumber
 
Applying the above inequality inductively  one gets   \eqref{eq:lya}.
\end{proof}
 
The last result we need is a bound of    $ \sum_{t=0}^k ({\CE^t})^2$ and $\sum_{t=0}^k ({\TE^t})^2$ in (\ref{eq:lya}) in terms of $\sum_{t=0}^k ({\NT^t})^2 ({\gamma^t})^2 $.  
\begin{lemma}\label{lm:sqauresum}
Define \vspace{-0.2cm}$$\varrho_{\bd{c}} \triangleq \frac{2C_2^2}{(1-\rho)^2}\quad \text{and}\quad \varrho_{\bd{t}} \triangleq \frac{36 \left( C_0 C_L  \right)^2 \left(2 C_2^2 + (1-\rho)^2 \right)}{(1-\rho)^4}.$$ 
The following holds:   $k \in \mathbb{N}$,\vspace{-0.2cm}\begin{equation}\label{lim:square:TE}
\begin{aligned}
  & \sum_{t=0}^k ({\CE^t})^2 \leq c_{\bd{c}} + \varrho_{\bd{c}} \sum_{t=0}^{k} ({\NT^t})^2 {(\gamma^t})^2, \\
  & \sum_{t=0}^k ({\TE^t})^2 \leq   c_{\bd{t}} + \varrho_{\bd{t}} \sum_{t=0}^{k}  ({\NT^t})^2 {(\gamma^t})^2,
\end{aligned}
\end{equation}
where  $c_{\bd{c}}$ and $c_{\bd{t}}$ are some positive constants. 
\end{lemma}

\begin{proof}
	The proof follows from Proposition \ref{lm:error_bounds}  and Lemma \ref{lm:lihua} below, which is a variant  of \cite{xu2015augmented} (its proof is thus omitted).
 \begin{lemma}\label{lm:lihua} Let $\{u^k\}_{k=0}^{\infty}, \{v^k_i\}_{k=0}^{\infty}, i = 1,\ldots, m$, be nonnegative sequences;   $\lambda \in (0,1)$; and        $R_0\in \mathbb{R}_+$ such that \vspace{-0.1cm}\begin{equation*}
u^{k+1} \leq R \lambda^{k} + \sum_{l=0}^{k} \lambda^{k-l} v^l.\vspace{-0.1cm}\end{equation*} 
   \vspace{-0.1cm}
 Then,  there holds:   $k\in \mathbb{N},$ \vspace{-0.1cm} $$\sum_{l=0}^{k} ({u^l})^2 \leq ({u^{0}})^2 + \frac{2 R^2}{1-\lambda^2}  + \frac{2}{(1-\lambda)^2} \sum_{l=0}^{k}  ({v^l})^2.$$
\end{lemma}
\end{proof}
Using \eqref{lim:square:TE} in  \eqref{eq:lya}, we   finally obtain\vspace{-0.2cm}
\begin{equation}\label{eq:lyapunov_eq1}
\sum_{t=0}^k ({\NT^t})^2 \gamma^t  ( \lbpsi -   {\gamma^t} C_4(\alpha,\beta) )\leq      F(\CV^0)  - F^{\inf} + C_5(\alpha,\beta)
\end{equation}
with 
$C_4(\alpha,\beta)\triangleq  ({1}/{2})\,(L + \alpha^{-1} + \beta^{-1}   + C_L^2 I \alpha \varrho_{\bd{c}}    + \eta^{-2} \beta \varrho_{\bd{t}})$ and 
 $C_5(\alpha,\beta)= ({1}/{2}) \left( C_L^2 I \alpha c_{\bd{c}}  +    \eta^{-2} \beta c_{\bd{t}} \right)$; and $F^{\inf}>-\infty$ is the lower bound of $F$.

We are now ready to prove Theorems ~\ref{thm:sublinear_const} and \ref{thm:dimi_sublinear}.

\subsection{Proof of Theorem~\ref{thm:sublinear_const}}
Set $\gamma^k \equiv \gamma$, for all $k\in \mathbb{N}_0.$ By \eqref{eq:lyapunov_eq1}, one infers that  $\sum_{t=0}^{\infty} {\NT^t}^2 <\infty$ if  $\gamma$ satisfies $0 < \gamma < \bar{\gamma}_2 (\alpha,\beta)$, with 
$\bar{\gamma}_2 (\alpha,\beta) \triangleq   \lbpsi / C_4(\alpha,\beta). $
 Note that $\bar{\gamma}_2(\alpha,\beta)$ is maximized setting  $\alpha=\alpha^\star=\left( C_L \sqrt{I \varrho_{\bd{c}}} \right)^{-1}$ and $\beta=\beta^\star=\eta \varrho_{\bd{t}}^{-{1}/{2}}$, resulting in   
 \begin{equation}\label{eq:gamma_2}
 \bar{\gamma}_2 (\alpha^\star,\beta^\star) =  (2 \lbpsi) /({ L + 2 C_L \sqrt{I \varrho_{\bd{c}}} + 2 \eta^{-1} \sqrt{\varrho_{\bd{t}}} }).
 \end{equation}

Let $0 < \gamma < \bar{\gamma}_2 (\alpha^\star,\beta^\star) .$  Given $\delta>0$,  let $T_{\delta}$ be the first iteration $k\in \mathbb{N}_0$ such that $M_F (\bx^k)  \leq \delta$.
Then we have 
\begin{equation*}
\begin{aligned}
&  T_{\delta} \cdot \delta   < \sum_{k=0}^{T_{\delta}-1} M_F(\bx^k)  \leq  \sum_{k=0}^{\infty} M_F(\bx^k) \\
& \overset{\eqref{lim:square:MF}}{\leq}   C_3 \sum_{k=0}^{\infty} ({\CE^k})^2  + 3 \eta^{-2} \sum_{k=0}^{\infty} \left( ({\TE^k})^2   + ({\NT^k})^2 \right)   \\
&  \!\!\!\overset{\eqref{lim:square:TE},\eqref{eq:lyapunov_eq1}}{\leq}  \frac{F(\CV^0) - F^{\inf} +  C_5(\alpha^\star,\beta^\star) }{ \gamma   ( \lbpsi -   {\gamma} C_4(\alpha^\star,\beta^\star) )}   \cdot C_6 + C_7 <\infty  
     \end{aligned}
\end{equation*}
where $C_6\triangleq   C_3 \varrho_{\bd{c}} (\gamma)^2 + 3 \eta^{-2} \left( \varrho_{\bd{t}} (\gamma)^2 +1 \right)     $ and $C_7$ is some constant.
  Therefore,   $T_\delta = \mathcal{O} (1/\delta)$.

\subsection{Proof of Theorem \ref{thm:dimi_sublinear}.}
 
We begin showing that the  step-size sequence $\{\gamma^t\}_{t\in \mathbb{N}_0}$ induced by the local step-size sequence $\{\alpha^t\}_{t\in \mathbb{N}_0}$ and the asynchrony mechanism satisfying  Assumption~\ref{ass:delays}  is  nonsummable. The proof is straightforward and is thus omitted.  \begin{lemma}\label{Lemma_gamma_global}
  Let $\{\gamma^t\}_{t\in \mathbb{N}_0}$ be the global step-size sequence resulted from   Algorithm~\ref{alg:AsyTracking}, under Assumption~\ref{ass:delays}. Then, there hold:    $  \lim_{t \to \infty}\gamma^t = 0$ and   $\sum_{t=0}^\infty \gamma^t =\infty$. \end{lemma}

Since  $  \lim_{t \to \infty}\gamma^t = 0$,    there exists a sufficiently large $k\in \mathbb{N}$, say $\bar{k}$, such that $\lbpsi -   {\gamma^k} C_4(\alpha^\star,\beta^\star)\geq \eta/2$ for all $k>\bar{k}$. It is not difficult to check that this, together with 
   \eqref{eq:lyapunov_eq1}, yields $\sum_{k=0}^\infty ({\NT^k})^2  \gamma^k <\infty$.
   We can then write

\begin{equation}\label{eq:dimi_summable}\hspace{-0.2cm}
\begin{aligned}
& \sum_{k=0}^{\infty} M_F(\bx^k) \gamma^k  \\&\overset{\eqref{lim:square:MF}}{\leq}   C_3 \sum_{k=0}^{\infty} ({\CE^k})^2 \gamma^k    + 3 \eta^{-2} \sum_{k=0}^{\infty} \left( ({\TE^k})^2  +  ({\NT^k})^2 \right) \gamma^k   < C_8, 
\end{aligned}
\end{equation}
for some finite constant $C_8$, where in the last inequality we used \eqref{lim:square:TE}, $\sum_{k=0}^\infty ({\NT^k})^2  \gamma^k <\infty$ and $  \lim_{t \to \infty}\gamma^t = 0$.

    Let $N_{\delta} \triangleq \inf  \big\{ k\in \mathbb{N}_0  \,: \,\sum_{ t =0}^{k} \gamma^t \geq C_8/\delta \big\}$. Note that $N_{\delta}$   exists, as $\sum_{ k =0}^{\infty} \gamma^k =\infty$ (cf. Lemma \ref{Lemma_gamma_global}). Let  $T_{\delta} \triangleq \inf  \big\{ k\in \mathbb{N}_0  \,: M_F(\bx^k)  \leq \delta\big\}$. 
    It must be $T_{\delta} \leq N_{\delta}$. 
   In fact, suppose by contradiction  that   $T_{\delta} > N_{\delta}$; and thus  $M_F(\bx^{k}) > \delta$, for  $0 \leq {k} \leq N_{\delta}$.  It would imply $\sum_{k=0}^{N_{\delta}} M_F(\bx^k) \gamma^k > \delta \sum_{k=0}^{N_{\delta}} \gamma^k  \geq \delta \cdot (C_8/ \delta) = C_8,$
which contradicts \eqref{eq:dimi_summable}. This proves (\ref{eq:sublinear-rate}). 

\vspace{-0.2cm}

 \section{Conclusions} \label{sec:conclude} 

We proposed ASY-SONATA, a distributed asynchronous algorithmic framework for convex and nonconvex (unconstrained, smooth) multi-agent problems, over digraphs. The algorithm is robust against uncoordinated agents' activation and (communication/computation) (time-varying) delays.  When employing a constant step-size, ASY-SONATA achieves a linear rate for strongly convex objectives--matching  the rate of a centralized gradient algorithm--and  sublinear rate for    (non)convex problems.  Sublinear  rate is also established when  agents employ    uncoordinated diminishing step-sizes, which is more realistic in a distributed setting.  
To the best of our knowledge,  ASY-SONATA is  the first distributed algorithm enjoying the above properties, in the general   asynchronous setting described in the paper. 
\vspace{-0.2cm}

\appendix

\subsection{Proof of Lemma~\ref{lm:scramb}}\label{pf:Mscramb}
We study any entry $\eM_{hm}^{k+K_1-1}$ with $m\in \widehat{\mathcal{V}}$ and $h\in \mathcal{V}$.  We prove the result by considering the following four cases.

\noindent \textbf{(i)} Assume $h=m\in \mathcal{V}.$  It is easy to check that $\eM_{hh}^k \geq \lbm$, for any $k\in \mathbb{N}_0$ and  $h\in \mathcal{V}$. Therefore,  $\eM_{hh}^{k+s-1:k} \geq \prod_{t=k}^{k+s-1} \eM_{hh}^{t} \geq \lbm^s$, for all $k,s \in \mathbb{N}_0$ and   $h\in \mathcal{V}$.

\noindent \textbf{(ii)} Let $(m,h) \in \mathcal{E}$; and let $s$ be the first time     $m$ wakes up in the  interval $[k,k+T-1]$. We have
$
\eM_{(m,h)^0, m}^{s}  = a_{hm}.
$
The information that node $m$ sent to node $(m,h)^0$ at iteration $s$ is received by node $h$ when the information is on some virtual node $(m,h)^d$.  We discuss separately the following three sub-cases for $d$: 1) $1\leq d \leq D-1$; 2) $d=0$; and 3) $d=D$.

\noindent {\it 1) $1\leq d \leq D-1$}: We have
\begin{equation*}
\begin{aligned}
&  \eM_{(m,h)^d, (m,h)^0}^{s+d : s+1} =   \eM_{(m,h)^d, (m,h)^{d-1}}^{s+d} \cdots  \eM_{(m,h)^1, (m,h)^0}^{s+1}  =  1,\\
&  \eM_{h,(m,h)^d}^{s+d+1} =  a_{hh}.
\end{aligned}
\end{equation*}

\noindent Therefore,
 $\eM_{hm}^{s+d+1:s} = \eM_{h,(m,h)^d}^{s+d+1} \eM_{(m,h)^d, (m,h)^0}^{s+d : s+1}$  $\eM_{(m,h)^0, m}^{s} $
 $= a_{hh} a_{hm} \geq \lbm^2$.

\noindent {\it 2) $d=0$}: We simply have   
\begin{equation*}
\begin{aligned}
\eM_{hm}^{s+1:s} = \eM_{h,(m,h)^0}^{s+1} \eM_{(m,h)^0, m}^{s} = a_{hh} a_{hm} \geq \lbm^2.
\end{aligned}
\end{equation*}
Therefore, for $0\leq d\leq D-1$, 
\begin{equation*}
\begin{aligned}
& \eM_{hm}^{k+2T+D-1:k}  = \eM_{hh}^{k+2T+D-1:s+d+2} \eM_{hm}^{s+d+1:s} \eM_{mm}^{s-1:k}  \\
& \geq \lbm^{k+2T+D - s - d -2} \lbm^{2} \lbm^{s-k}  \geq \lbm^{2T+D}.
\end{aligned}
\end{equation*}

\noindent {\it 3) $d=D$}:  Before agent $j$ wakes up at time $s+D+\tau$, where $1\leq \tau \leq T$, the information will stay on virtual nodes $(m,h)^D.$  Once agent $j$ wakes up, nodes $(m,h)^D$ will send all its information to it.  Then we have 
\begin{equation*}
\begin{aligned}
 \eM_{(m,h)^D, (m,h)^0}^{s+D : s+1} =  1,  \quad \eM_{h, (m,h)^D}^{s+D+\tau  : s+D+1} =    a_{hh}. 
\end{aligned}
\end{equation*}
Similarly, we have  
\begin{equation*}
\begin{aligned}
& \eM_{hm}^{k+2T+D-1}  =  \eM_{hh}^{k+2T+D-1: s+D+\tau +1}\eM_{h, (m,h)^D}^{s+D+\tau  : s+D+1}  \\
& \cdot  \eM_{(m,h)^D, (m,h)^0}^{s+D : s+1}  \eM_{(m,h)^0, m}^{s} \eM_{mm}^{s-1:k}  
 \geq  \lbm^{2T+D}.
\end{aligned}
\end{equation*}
To summarize,  in all of the three sub-cases, we have
\begin{equation*}
\begin{aligned}
&  \eM_{hm}^{k+K_1-1} \geq \eM_{hh}^{k+K_1-1: k+2T+D} \eM_{hm}^{k+2T+D-1:k}  \\
& \geq \lbm^{K_1-2T-D} \lbm^{2T+D} = \lbm^{K_1 }.
\end{aligned}
\end{equation*}

\noindent \textbf{(iii)} Let $m \neq h $ and $(m,h) \in \mathcal{V} \times \mathcal{V} \setminus \mathcal{E}$.  Since the graph $(\mathcal{V},\mathcal{E})$ is connected, there are mutually different agents $i_1,\ldots,i_r$, with $r\leq I-2$, such that $(m,i_1),(i_1,i_2),\ldots,(i_{r-1},i_{r}),(i_r,h)\subset \mathcal{E},$
which is actually a directed path from $m$ to $h$ in the graph $(\mathcal{V},\mathcal{E})$.  Then, by result proved in (ii), we have
\begin{equation*}
\begin{aligned}
& \eM_{hm}^{k+(I-1)(2T+D)-1:k} \geq   \eM_{hh}^{k+(I-1)(2T+D)-1:k+(r+1)(2T+D)}  \\  
&\cdot \eM_{hi_r}^{k+(r+1)(2T+D)-1:k+r(2T+D)}    \cdots  \eM_{i_1 m}^{k+2T+D-1:k} \\
 &\geq  \lbm^{(I-r-2)(2T+D)} \lbm^{(r+1)(2T+D)} =  \lbm^{(I-1)(2T+D)}.
\end{aligned}
\end{equation*}
We can then  easily get  
$ \eM_{hm}^{k+K_1-1:k}    =  \eM_{hh}^{k+K_1-1: k+(I-1)(2T+D)} \eM_{hm}^{k+(I-1)(2T+D)-1:k} $ $  \geq 
  \lbm^{K_1}$.

\noindent \textbf{(iv)} If $m$ is a virtual node, it must be associated with an edge $(j,i) \in \mathcal{E}$ and there exists $0\leq d \leq D$ such that $m = (j,i)^d.$  
A similar argument as in (ii) above  shows that any information on $m$ will eventually enter node $i$ taking   $1\leq \tau \leq D+T$.  That is,    
$
\eM_{im}^{k+\tau-1:k}  =a_{ii},
$ for some $1\leq \tau \leq D+T$.
On the other hand, by the above results, we know 
\begin{equation*}
\begin{aligned}
\eM_{hi}^{k+T+D+(I-1)(2T+D)-1:k+T+D}  \geq \lbm^{(I-1)(2T+D)}.
\end{aligned}
\end{equation*}
Therefore,
\begin{equation*}
\begin{aligned}
& \eM_{hm}^{k+K_1-1:k} 
\geq  \eM_{hi}^{k+K_1-1:k+T+D}    \eM_{ii}^{k+T+D-1:k+\tau}  \eM_{im}^{k+\tau-1:k} \\
&  \geq \lbm^{(I-1)(2T+D)} \lbm^{T+D-\tau} \lbm \geq  \lbm^{K_1}
\end{aligned}
\end{equation*} \qed\vspace{-0.4cm}

\subsection{Proof of Lemma~\ref{conv_lm}}\label{pf:conv}
Fix $N\in \mathbb{N}$, and let $k$ such that $1 \leq k+1 \leq N$. We have:  
\begin{equation*}
\begin{aligned}
  \frac{u^{k+1}}{\lambda^{k+1}} &\leq  \frac{R_0 }{\lambda} \left( \frac{\lambda_0}{\lambda} \right)^k + \sum_{i = 1}^m \frac{R_i}{\lambda}  \sum_{l = 0}^k \left( \frac{\lambda_i}{\lambda} \right)^{k-l} \frac{v_i^l}{\lambda^l} \\
& 					\leq \frac{R_0 }{\lambda}  +  \sum_{i = 1}^m \frac{R_i}{\lambda}  \abs{v_i }^{\lambda,N} \sum_{l = 0}^k \left( \frac{\lambda_i}{\lambda} \right)^{k-l}    \\
&\leq \frac{R_0 }{\lambda}  +   \sum_{i = 1}^m \frac{R_i}{\lambda-\lambda_i}  \abs{v_i}^{\lambda,N}.
\end{aligned}\vspace{-0.3cm}
\end{equation*}
Hence, \vspace{-0.3cm}
 
\begin{align*}
   \abs{u}^{\lambda,N} 
& 			\leq  \max\left( u_0 , \frac{R_0 }{\lambda}  +   \sum_{i = 1}^m \frac{R_i}{\lambda-\lambda_i}  \abs{v_i}^{\lambda,N}\right) \\
& 			\leq   u^0 +  \frac{R_0 }{\lambda}  +   \sum_{i = 1}^m \frac{R_i}{\lambda-\lambda_i}  \abs{v_i}^{\lambda,N}.\qed
\end{align*} \vspace{-0.6cm}

\subsection{Proof of Theorem~\ref{thm:SGT}}\label{pf:sgt}
From   \cite[Ch. 5.6]{horn1990matrix}, we know that if $\rho(\bd{T})<1,$ then $\lim_{k\to \infty} \bd{T}^k =0$, the series $\sum_{k=0}^{\infty} \bd{T}^k$ converges (wherein we define $\bd{T}^0 \triangleq \bd{I}$), $\bd{I}-\bd{T}$ is invertible and $\sum_{k=0}^{\infty} \bd{T}^k =  (\bd{I}-\bd{T})^{-1}.$

Given $N\in \mathbb{N}$, using  \eqref{eq:SGT}  recursively,  yields: 
$\bd{u}^{\lambda, N} \leq  \bd{T} \bd{u}^{\lambda, N} + \bm{\beta}  \leq   \bd{T} \left( \bd{T} \bd{u}^{\lambda, N} + \bm{\beta} \right) + \bm{\beta} 
 =  \bd{T}^2 \bd{u}^{\lambda, N} + \left (\bd{T} + \bd{I} \right) \bm{\beta} \leq 
  \cdots  
 \leq  \bd{T}^{\ell} \bd{u}^{\lambda, N} + \sum_{k=0}^{\ell -1} \bd{T}^k \bm{\beta},$
for any   $\ell\in   \mathbb{N}$.  Let $\ell \to \infty$, we get $\bd{u}^{\lambda, N} \leq (\bd{I}-\bd{T})^{-1} \bm{\beta}.$  Since  this holds for any given   $N\in \mathbb{N}$,   we have 
$\bd{u}^{\lambda} \leq (\bd{I}-\bd{T})^{-1} \bm{\beta}.$ Hence, $\bd{u}^{\lambda}$ is bounded, and thus  each $u_i^k$ vanishes at an   R-linear rate $\mathcal{O}(\lambda^k)$.\qed\vspace{-0.3cm}

\subsection{Proof of the rate decay \eqref{eq:rate-expression} (Theorem \ref{thm:linear_const}) }\label{sec:linear_2}
 
Let   
$\lambda \geq \mathcal{L}(\gamma) + \epsilon \gamma$,
with  $\epsilon>0$ to be properly chosen.  Then,  
\begin{equation}\label{eq:suff_lambda}
\begin{aligned}
& \mathfrak{B}(\lambda;\gamma)  \leq   \left( 1+ \frac{L }{\epsilon}  \right) \frac{b_2 \gamma}{\lambda-\rho} \\ 
& + \left(   \left( 1+ \frac{L }{\epsilon}  \right) \frac{b_2}{\lambda-\rho} + b_1 + \frac{L b_2 }{ \epsilon }  \right)  \frac{C_2 \gamma}{\lambda- \rho} .
\end{aligned}
\end{equation}
Using  $\lambda - \rho <1$, a sufficient condition for Eq.~\eqref{eq:rate_func} is  [RHS less than one]
\begin{equation}\label{eq:TR_assp2}
 \left( b_1 C_2  + \frac{L b_2 C_2  }{ \epsilon }    + \left( 1+ \frac{L }{\epsilon}  \right) b_2 (1 + C_2) \right) \gamma \leq \left( \lambda - \rho \right)^2.
\end{equation}
Now set  $\epsilon =  ({\tau \lbpsi^2})/{2}$.  Since the RHS of the above inequality can be arbitrarily close to $(1-\rho)^2$,  an upper bound of $\gamma$ is 
\begin{align*}
& \hat{\gamma}_2 \triangleq \\
& (1-\rho)^2 \bigg/  \underbrace{\left( b_1 C_2  + \frac{2 L b_2 C_2  }{ \tau \lbpsi^2 }    + \left( 1+ \frac{2 L }{\tau \lbpsi^2}  \right) b_2 (1 + C_2)  \right) } _{\triangleq J_1} .
\end{align*}
According to $\lambda \geq \mathcal{L}(\gamma) + \epsilon \gamma$ and \eqref{eq:TR_assp2}, we get 
\begin{align}
\lambda = \max \left(1 -  \frac{\tau \lbpsi^2 \gamma}{2}, \quad \rho + \sqrt{ J_1 \gamma} \right).
\end{align}
Notice that when $\gamma$ goes from $0$ to $ \hat{\gamma}_2$, the first argument inside the $\max$ operator decreases from $1$ to $1 -   ({\tau \lbpsi^2 \hat{\gamma}_2})/{2}$, while the second argument increases from $\rho$ to $1$.  Letting 
$1 -  \frac{\tau \lbpsi^2 \gamma}{2} =   \rho + \sqrt{ J_1 \gamma} ,$
we get the solution as $\hat{\gamma}_1 = \left(\frac{\sqrt{J_1 + 2 \tau\eta^2 (1-\rho)} - \sqrt{J_1}}{\tau\eta^2} \right)^2.$  The expression of   $\lambda$ as in \eqref{eq:rate-expression} follows readily.\qed 

\subsection{Proof of Corollary~\ref{cor:signal}}\label{pf:signal}
We know that $\mathfrak{m}^{k}_z  =   \sum_{i=1}^I {{z}_i^0} + \sum_{t=0}^{k-1} \epsilon^t .$  Clearly $\mathfrak{m}^0_z = \sum_{i=1}^I {{z}_i^0} =  \sum_{i = 1}^I \tilde{{u}}_{i}^{0}.$  Suppose for $k = \ell$, we have that $\mathfrak{m}^\ell_z =  \sum_{i = 1}^I \tilde{{u}}_{i}^\ell.$  Then we have that
\begin{align*}
&\mathfrak{m}^{\ell +1}_z =  \mathfrak{m}^{\ell }_z + \epsilon^\ell = \left( \sum_{i = 1}^I \tilde{{u}}_{i}^\ell \right) + {u}_{i^\ell}^{\ell+1} - \tilde{{u}}_{i^\ell}^{\ell}  \\
 & =  \sum_{j \neq i^\ell} \tilde{{u}}_{j}^\ell +  {u}_{i^\ell}^{\ell+1} = \sum_{i = 1}^I \tilde{{u}}_{i}^{\ell+1}.
\end{align*}
Thus we have that $\mathfrak{m}_z^{k}= \sum_{i = 1}^I \tilde{{u}}_{i}^{k}, \quad \forall k \in \mathbb{N}_0.$

Now we assume that $\lim_{k\to\infty} \sum_{i=1}^I \abs{{u}_i^{k+1}-{u}_i^{k}} = 0.$  Notice that for $k\geq T,$
\begin{align*}
\abs{\epsilon^k} = & \abs{{u}_{i^k}^{k+1} - \tilde{{u}}_{i^k}^{k}} \leq \sum_{t=k-T+1}^k \abs{{u}_{i^k}^{t+1} - {u}_{i^k}^{t}} \\
 \leq & \sum_{t=k-T+1}^k \sum_{i=1}^I \abs{{u}_{i}^{t+1} - {u}_{i}^{t}}.
\end{align*}
Therefore we have that $\lim_{k\to\infty} \abs{\epsilon^k} = 0.$  According to Theorem~\ref{thm:track} and \cite[Lemma~7(a)]{di2016next}, we have that 
$$\lim_{k\to\infty}\abs{\tr_i\spe{k+1} -  \left( 1/I \right) \cdot \sum_{i = 1}^I \tilde{{u}}_{i}^{k+1} } = 0.$$
On the other hand, we have that
\begin{align*}
& \abs{\sum_{i = 1}^I {u}_{i}^{k+1}-\sum_{i = 1}^I \tilde{{u}}_{i}^{k+1} } \leq \sum_{i = 1}^I \abs{{u}_{i}^{k+1} - \tilde{{u}}_{i}^{k+1} } \\
& \leq \sum_{i = 1}^I \sum_{t=k-T+1}^k   \abs{{u}_{i}^{t+1} - {u}_{i}^{t}} \overset{k\to\infty}{\longrightarrow} 0.
\end{align*}
By triangle inequality, we get that
$$\lim_{k\to\infty}\abs{\tr_i\spe{k+1} -  \left( 1/I \right) \cdot \sum_{i = 1}^I {u}_{i}^{k+1} } = 0.$$
\qed

\subsection{Proof of Lemma~\ref{lm:polynomial}}\label{pf:polynomial}
``$\Longleftarrow:$'' From $p(1)>0,$ we know that $\sum_{i=1}^m a_i <1.$  We prove by contradiction.  Suppose there is a root $z_*$ of $p(z)$ satisfying $\abs{z_*} \geq 1$, then we have
$$z_*^m = a_1 z_*^{m-1} + a_2 z_*^{m-2} + \ldots + a_{m-1} z_* + a_m.$$
Clearly $z_* \neq 0$, so equivalently
$$1 = a_1 \frac{1}{z_*} + a_2 \frac{1}{z_*^2} + \ldots + a_{m-1} \frac{1}{z_*^{m-1}} + a_m \frac{1}{z_*^{m}}.$$
Further,
\begin{align*}
& 1 =  \left | a_1 \frac{1}{z_*} + a_2 \frac{1}{z_*^2} + \ldots + a_{m-1} \frac{1}{z_*^{m-1}} + a_m \frac{1}{z_*^{m}} \right| \\
&\leq  a_1 \frac{1}{\abs{z_*}} + a_2 \frac{1}{\abs{z_*}^2} + \ldots + a_{m-1} \frac{1}{\abs{z_*}^{m-1}} + a_m \frac{1}{\abs{z_*}^{m}}  \\
&\leq  a_1 + a_2 + \ldots + a_{m-1} + a_m \\
&<  1.
\end{align*}
This is a contradiction.

``$\Longrightarrow:$'' If $p(1) = 0,$ we clearly have that $z_p \geq 1.$  Now suppose $p(1) < 0.$  Because $\lim_{z\in \mathbb{R}, z\to +\infty} p(z) = + \infty$ and $p(z)$ is continuous on $\mathbb{R}$, we know that $p(z) $ has a zero in $(1,+\infty) \subset \mathbb{R}.$  Thus $z_p > 1.$
\qed

\subsection{Proof of Lemma~\ref{lm:Wscramb}}\label{pf:wscramb}

We interpret the dynamical system \eqref{eq:chi} over an augmented graph.  We begin constructing the augmented graph obtained adding virtual nodes to the original graph $\mathcal{G}=(\mathcal{V},\mathcal{E})$.  We associate each node $i \in \mathcal{V}$ with an ordered set of virtual nodes $i[0],i[1],\ldots, i[D]$; see Fig.~\ref{fig:static}.  We still call the nodes in the original graph $G$ as \textit{computing agents} and the virtual nodes as \textit{noncomputing agents}.  We now identify the neighbors of each agent in this augmented system.  Any noncomputing agent $i[d]$, $d = D, D-1, \cdots, 1$, can only receive information from the previous virtual node $i[d-1]$; $i[0]$ can only receive information from the real node $i$ or simply keep its value unchanged; computing agents cannot communicate among themselves.

\begin{figure}[H]
\centering
\includegraphics[scale=0.3]{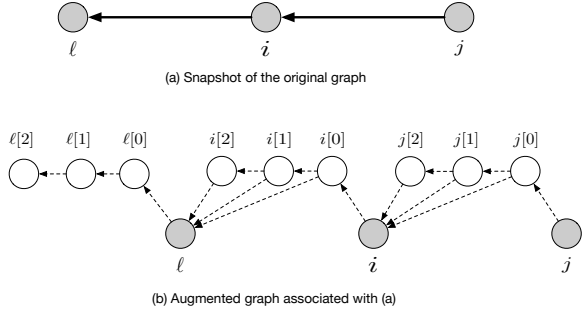}
\caption{Example of augmented graph, when the maximum delay is $D=2$; three noncomputing agents are added for each node $i\in \mathcal{V}$.}
\label{fig:static}
\end{figure}


At the beginning of each iteration $k$, every computing agent $i\in \mathcal{V}$ will store the information $x_i^k$;  whereas every noncomputing agent $i[d]$, with $d = 0,1,\cdots,D,$ will store the delayed information $v_i^{k-d}$.  The dynamics over the augmented graph happening in iteration $k$ is described by \eqref{eq:chi}.  In words, any noncomputing agent $i[d]$ with $i\in \mathcal{V}$ and $d = D,D-1,\cdots,1$ receives the information from $i[d-1]$; the noncomputing agent $i^k[0]$ receives the perturbed information $x_{i^k}^k + \delta^k$ from node $i^k$; the values of noncomputing agents $j[0]$ for $j \in \mathcal{V} \setminus \{i^k\}$ remain the same; node $i^k$ sets its new value as a weighted average of the perturbed information $x_{i^k}^k + \delta^k$  and $v_{j}^{k - d_{j}^k}$'s received from the virtual nodes $j[d_{j}^k]$'s for $ j\in \mathcal{N}_{i^k}^{\text{in}};$  and the values of the other computing agents remain the same.  The dynamics is further illustrated in Fig.~\ref{fig:active}.  The following Lemma shows that the product of a sufficiently large number of any instantiations of the matrix $\W^k$, under Assumption~\ref{ass:delays}, is a scrambling matrix.

\begin{figure}[H]
\centering
\includegraphics[scale=0.3]{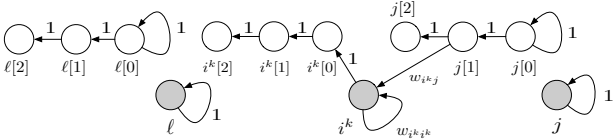}
\caption{The dynamics in iteration $k.$  Agent $i^k$ uses the delayed information $v_j^{k-1}$ from the virtual node $j[1].$}
\label{fig:active}
\end{figure}

\begin{lemma}\label{lm:Wscrambling}
Let $\{\W^k \}_{k\in \mathbb{N}_0}$ be the sequence of augmented matrices generated according to the dynamical system \eqref{eq:PAC}, under Assumption~\ref{ass:delays}, and with   $\mathbf{W}$ satisfying Assumption~\ref{assumption_weights} (i), (ii).  Then, for any $k\in \mathbb{N}_0$, $\W^k$ is row stochastic and $\W^{k+K_1-1:k}$ has the property that all entries of its first $I$ columns are uniformly lower bounded by $\eta $.
\end{lemma}

\begin{proof}
We study any entry $\eW_{hm}^{k+K_1-1}$ with $m\in \mathcal{V}$ and $h\in \widehat{\mathcal{V}}$.  We prove the result by considering the following four cases.

\noindent \textbf{(i)} Assume $h=m\in \mathcal{V}.$  Since $\eW_{hh}^k \geq \lbm$ for any $k\in \mathbb{N}_0$ and any $h\in \mathcal{V}$, we have $\eW_{hh}^{k+s-1:k} \geq \prod_{t=k}^{k+s-1} \eW_{hh}^{t} \geq \lbm^s$ for $\forall \, k \in \mathbb{N}_0, \,\forall \, s \in \mathbb{N}$ and $\forall \, h\in \mathcal{V}$.

\noindent \textbf{(ii)} Assume that $(m,h) \in \mathcal{E}.$  Suppose that the first time when agent $h$ wakes up during the time interval $[k + T + D,k + 2T + D-1]$ is $s$, and agent $h$ uses the information $v_m^{s-d}$ from the noncomputing agent $m[d].$  Then we have
$$
\eW_{h,m[0]}^{s:s-d}  \geq  \eW_{h,m[d]}^{s}  \cdots  \eW_{m[1],m[0]}^{s-d}  = w_{hm} \geq \lbm^{d+1}.
$$
Then suppose that the last time when agent $m$ wakes up during the time interval $[s-d-T,s-d-1]$ is $s-d-t$.  The noncomputing agent $m[0]$ receives some perturbed information from agent $m$ at iteration $s-d-t$ and then performs self-loop (i.e., keep its value unchanged) during the time interval $[s-d-t+1,s-d-1]$.  Thus we have 
$$
 \eW_{m[0],m}^{s-d-1:s-d-t} = \eW_{m[0],m[0]}^{s-d-1:s-d-t+1} \eW_{m[0],m}^{s-d-t}  =  1\cdot 1 \geq \lbm^{t-1}.
$$
Therefore we have 
\begin{align*}
& \eW_{hm}^{k + 2T + D-1:k}  \geq  \eW_{hh}^{k + 2T + D-1:s+1} \eW_{h,m[0]}^{s:s-d}  \\
&\cdot  \eW_{m[0],m}^{s-d-1:s-d-t}  \eW_{mm}^{s-d-t-1:k} \\
& \geq  \lbm^{k+2T+D-s-1} \lbm^{d+1} \lbm^{t-1} \lbm^{s-d-t-k}  \geq \lbm^{2T+D}.
\end{align*}
Further we have 
\begin{align*}
& \eW_{hm}^{k+K_1-1} \geq \eW_{hh}^{k+K_1-1: k+2T+D} \eW_{hm}^{k+2T+D-1:k}  \\
& \geq \lbm^{K_1-2T-D} \lbm^{2T+D} = \lbm^{K_1 }.
\end{align*}

\noindent \textbf{(iii)} Assume that $m \neq h $ and $(m,h) \in \mathcal{V} \times \mathcal{V} \setminus \mathcal{E}$.  Because the graph $(\mathcal{V},\mathcal{E})$ is connected, there are mutually different agents $i_1,\ldots,i_r$ with $r\leq I-2$ such that $$(m,i_1),(i_1,i_2),\ldots,(i_{r-1},i_{r}),(i_r,h)\subset \mathcal{E},$$
which is actually a directed path from $m$ to $h$.  Then, by result proved in (ii), we have
\begin{equation*}
\begin{aligned}
& \eW_{hm}^{k+(I-1)(2T+D)-1:k} 							=  \eW_{hh}^{k+(I-1)(2T+D)-1:k+(r+1)(2T+D)}  \\
							  & \eW_{hi_r}^{k+(r+1)(2T+D)-1:k+r(2T+D)}\cdots  \eW_{i_2 i_1}^{k+2(2T+D)-1:k+2T+D}  \\
							  & \cdot \eW_{i_1 m}^{k+2T+D-1:k} \\
& \geq  \lbm^{(I-r-2)(2T+D)} \lbm^{(r+1)(2T+D)}=   \lbm^{(I-1)(2T+D)}.
\end{aligned}
\end{equation*}
Then we can easily get
\begin{equation*}
\begin{aligned}
& \eW_{hm}^{k+K_1-1:k}  =  \eW_{hh}^{k+K_1-1: k+(I-1)(2T+D)}  \eW_{hm}^{k+(I-1)(2T+D)-1:k} \\
& \geq  \lbm^{K_1-(I-1)(2T+D)} \lbm^{(I-1)(2T+D)} = \lbm^{K_1}.
\end{aligned}
\end{equation*}

\noindent \textbf{(iv)} If $h$ is a noncomputing node, it must be affiliated with a computing agent $j \in \mathcal{V}$, i.e., there exists $0\leq d \leq D$ such that $h = j[d].$  Then we have 
$$\eW_{h,j[0]}^{k+K_1-1:k+K_1-d} = \eW_{j[d],j[d-1]}^{k+K_1 -1}   \cdots \eW_{j[1],j[0]}^{k+K_1 -d}  = 1.$$
Suppose that the last time when agent $j$ wakes up during the time interval $[k+K_1 -d - T, k+K_1 -d - 1]$ is $s$.  We have
$$\eW_{j[0],j}^{k+K_1 -d -1 : s} = \eW_{j[0],j[0]}^{k+K_1 -d -1 : s+1} \eW_{j[0],j}^{s} = 1.$$
By results proved before, we have
\begin{align*}
 &  \eW_{hm}^{k+K_1  -1 : k} \geq \eW_{h,j[0]}^{k+K_1  -1 : k+K_1  -d} \eW_{j[0],j}^{k+K_1 -d -1 : s}  \\
& \cdot \eW_{jj}^{s-1 : k+ (I-1)(2T+D)} \eW_{jm}^{k+ (I-1)(2T+D)-1:k}  \\
& \geq  1 \cdot 1 \cdot \lbm^{s-k- (I-1)(2T+D)} \lbm^{(I-1)(2T+D)} \geq \lbm^{K_1}.
\end{align*}
\end{proof}
Based on Lemma~\ref{lm:Wscrambling},  we get the following result according to the discussion in \cite{nedic2010convergence}.

\begin{lemma}
In the setting above, there exists a sequence of stochastic vectors $\{\bm{\psi}^k\}_{k\in \mathbb{N}_0}$ such that for any $k \geq t\geq 0$,
$$\norm{\W^{k:t} -\bd{1}{\bm{\psi}^t}^{\top}}_{\infty} \leq \frac{2  (1+ \lbm^{-K_1})}{1- \lbm^{-K_1}} \rho^{k-t}.$$ 
Furthermore, $\psi_i^k\geq \lbpsi = \lbm^{K_1}$ for all $k\geq 0$ and  $i\in \mathcal{V}$.
\end{lemma}

The above result leads to Lemma~\ref{lm:Wscramb} by noticing that
$$\norm{\W^{k:t} -\bd{1}{\bm{\psi}^t}^{\top}} \leq \sqrt{(D+2)I} \norm{\W^{k:t} -\bd{1}{\bm{\psi}^t}^{\top}}_{\infty} \leq C_2 \rho^{k-t}.$$

\begin{footnotesize}
\bibliographystyle{IEEEtran}
\bibliography{reference}	
\end{footnotesize}

\end{document}